\documentclass[a4paper,11pt]{article}
\usepackage{amsmath,amssymb,amsfonts,graphicx,psfrag,subfig, amsthm}
\usepackage[latin1]{inputenc}

\newtheorem{thrm}{Theorem}[section]
\newtheorem{prpstn}[thrm]{Proposition}
\newtheorem{lmm}[thrm]{Lemma}
\newtheorem{dfntn}[thrm]{Definition}

\newtheorem{rmrk}[thrm]{Remark}

\newtheorem{crllr}[thrm]{Corollary}

\newcommand{\mb}{\medskip\noindent}
\newcommand{\gb}{\bigskip\noindent}
\def \ll {\langle}
\def \rr {\rangle}

\def \I {\{1,\, \dots,p\}}
\def \P {\textsc{P}}
\newcommand{\R}{\mathbb R}
\def \NNN {\mathcal{N}}
\def \CCC {\mathcal{C}}
\def \ee {\mathrm{e}}
\def \hh {{\bf h}}
\def \pp {{\bf p}}
\def \qqq {{\bf q}}
\def \qq {\mathrm{q}}
\def \uu {{\bf u}}

\def \ff {{\bf U}}

\def \WW {{\bf W}}
\def \FF {{\bf F}}
\def \ff {{\bf f}}
\def \GG {{\bf G}}

\def \UU {{\bf U}}
\def \ww {{\bf w}}
\def \vv {{\bf v}}
\def \xx {{\bf x}}
\def \yy {{\bf y}}
\def \zz {{\bf z}}
\def \NN {\mathrm{N}}
\def \Qc {\tilde{Q}}

\def \bflambda {{\boldsymbol{\lambda}}}
\def \bfmu {\boldsymbol{\mu}}
\def \bfxi {\boldsymbol{\xi}}

\def \virg {\, , \,\,}
\def \dsp {\displaystyle}
\def \vsp {\vspace{6pt}}
\def \noi {\noindent}

\def\sqw{\hbox{\rlap{\leavevmode\raise.3ex\hbox{$\sqcap$}}$%
\sqcup$}}
\def\findem{\ifmmode\sqw\else{\ifhmode\unskip\fi\nobreak\hfil
\penalty50\hskip1em\null\nobreak\hfil\sqw
\parfillskip=0pt\finalhyphendemerits=0\endgraf}\fi}

\title{Numerical scheme for a whole class of sweeping process}

\date{\today }
\author{Juliette Venel\\ LAMAV \\ Universit\'e de Valenciennes et du Hainaut-Cambr\'esis\\Mont Houy 59313 Valenciennes Cedex 9
\\juliette.venel@univ-valenciennes.fr}

\begin{document}

\maketitle
\tableofcontents

\begin{abstract}
The aim of this paper is to study a whole class of first order differential inclusions, which fit into the framework of perturbed sweeping process by uniformly prox-regular sets. After obtaining well-posedness results, we propose a numerical scheme based on a prediction-correction algorithm and we prove its convergence.
 Finally we apply these results to a problem coming from the modelling of crowd motion.
\end{abstract}

\mb {\bf Key-words:} Differential inclusions - Proximal normal cone - Uniform prox-regularity - Numerical analysis - Prediction-correction algorithm. 

\mb {\bf MSC:} 34A60, 65L20.

\gb

\section{Introduction}
The study of first order differential inclusions started in the 1960s with the theory of the maximal monotone operators (see e.g.~\cite{Brezis}). Later J.-J. Moreau considered a problem involving a time-dependent multivalued operator in~\cite{Moreausweep}. He dealt with the first \textit{sweeping process} by convex sets $ C(t)$ included in a Hilbert space:
\begin{equation} \frac{d\qqq}{dt}(t) \in -\partial I_{C(t)} (\qqq(t)),
 \label{moreausweep}
\end{equation}
 where $\partial I_{C}$ is the subdifferential of the characteristic function of a convex set $C $. 
Such a situation may be visualized as a point $\qqq(t)$ moving inside $C(t) $ and being pushed by the boundary of this convex set when contact is established.
This problem can also be written :
$$\frac{d\qqq}{dt}(t) \in -\NN(C(t),\qqq(t)),$$ where $\NN(C,\qqq) $ is the proximal normal cone to $C$ at $\qqq $ (see Definition~\ref{def:N}). He developped a so called \textit{catching-up algorithm} to build discretized solutions and so proved the well-posedness of (\ref{moreausweep}) under some assumptions on the set-valued map $C(\cdot) $. More precisely, in considering some subdivision $(J_k)_k$ of the time-interval, the set-valued map $C $ is approached by a piecewise constant multifunction taking value $C_k $ on $J_k$. The associated discretized solution $\tilde{\qqq} $ defined by $$ \forall t \in J_{k+1} \virg   \tilde{\qqq}(t)=\tilde{\qqq}_{k+1} = \P_{C_{k+1}}(\tilde{\qqq}_{k}),$$ 
with $\tilde{\qqq}_0 $ fixed to the initial value, converges to the unique solution of (\ref{moreausweep}).
\bigskip 
 
\noi Since then, important improvements have been provided by weakening the convexity assumption and by considering a perturbed version of this problem:
\begin{equation} \frac{d\qqq}{dt}(t) + \NN(C(t),\qqq(t)) \ni \ff(t,\qqq(t)). 
 \label{spp}
\end{equation}

\noi In~\cite{Valadier}, M. Valadier studied sweeping process by complements of convex sets in finite dimension without perturbation. Perturbations (even multivalued perturbations) have been taken into account by C. Castaing, T.X. D\'uc H\={a}, M.D.P Monteiro Marques and M. Valadier in~\cite{CDV,Castaingper}. \\

\noi Then the main concept, which appeared to weaken the convexity assumption of sets $C(t)$, is the notion of \textit{uniform prox-regularity}. A set $C$ is said to be {\it $\eta$-prox-regular} (or uniformly prox-regular with constant $ \eta$) if the projection onto $C$ is single-valued and continuous at any point whose the distance to $C$ is smaller than $\eta$. Under this assumption, the sweeping process without perturbation was firstly treated by G.~Colombo, V.V.~Goncharov in \cite{Colombo}, by H.~Benabdellah in \cite{Benab} and later by L.~Thibault in \cite{Thibsweep} and by G. Colombo, M.D.P.~Monteiro Marques in \cite{Monteiro}. The perturbed problem was later studied by M.~Bounkhel, J.-F.~Edmond and L.~Thibault in \cite{Thibnonconv, Thibsweep, Thibrelax, Thibbv}. \\

\noi In this paper, we consider perturbed sweeping process by a set-valued map $Q:[0,T]\rightrightarrows \R^d$ satisfying that for every $t\in [0,T]$, $Q(t) $ is the intersection of complements of smooth convex sets.
Let us first specify the set-valued map $Q$.
For $i \in \I$, let $g_i: [0,T] \times \R^d \rightarrow \R$ be a convex function with respect to the second variable. 
For every $t \in [0,T]$, we introduce the sets $Q_i(t)$ defined by:
\begin{equation}
 Q_i(t):=\left\{ \qqq \in \R^d \virg g_i(t,\qqq) \geq 0 \right\},
\label{def:Qi}
\end{equation}
and the feasible set $Q(t) $ (supposed to be nonempty) is
\begin{equation}
 Q(t):=\bigcap_{i=1}^{p} Q_i(t).
\label{def:Q}
\end{equation}
The associated perturbed sweeping process can be expressed as follows:
\begin{equation}
\left\{
\begin{array}{l}
\dsp \frac{d\qqq}{dt}(t) + \NN(Q(t),\qqq(t)) \ni \ff(t,\qqq(t)) \ \textmd{for a.e. } t \in [0,T] \vsp \\
\qqq(0)=\qqq_0 \in Q(0).
\end{array}
\right.
\label{incldiff}
\end{equation}

\noi This differential inclusion can be thought as follows: the point $\qqq(t)$, submitted to the perturbation $\ff(t,\qqq(t))$, has to stay in the feasible set $Q(t)$.
To obtain well-posedness results for (\ref{incldiff}), we will make the following assumptions which ensure the uniform prox-regularity of $Q(t)$ for all $t \in [0,T] $. 
We suppose there exists $c  >0$ and for all $ t$ in $[0,T]$ open sets $U_i(t) \supset Q_i(t) $ verifying
\begin{equation}
 \tag{A0}
d_H(Q_i(t), \R^d \setminus U_i(t)) > c,
\label{Ui}
\end{equation}
where $d_H $ denotes the Hausdorff distance.
Moreover we assume there exist constants $\alpha, \beta, M >0$ such that for all $ t$ in $[0,T]$, $g_i(t,\cdot)$ belongs to $ C^2(U_i(t))$ and satisfies 
\begin{equation}
 \tag{A1}
\forall \, \qqq \in U_i(t) \virg \alpha \leq |\nabla_{\qqq} g_i (t,\qqq) | \leq \beta,
\label{gradg}
\end{equation}
\begin{equation}
 \tag{A2}
\forall \, \qqq \in \R^d \virg|\partial_{t} g_i (t,\qqq) | \leq \beta,
\label{dtg}
\end{equation}
\begin{equation}
 \tag{A3}
\forall \, \qqq \in U_i(t) \virg  |\partial_{t}\nabla_{\qqq} g_i (t,\qqq) | \leq M,
\label{dtgradg}
\end{equation}
and 
\begin{equation}
 \tag{A4}
\forall \, \qqq \in U_i(t) \virg  |\mathrm{D}_\qqq^2 g_i (t,\qqq) | \leq M.
\label{hessg}
\end{equation}
For all $t\in [0,T]$ and $\qqq \in Q(t)$, we denote by $I(t,\qqq)$ the active set at $ \qqq$
\begin{equation}
 I(t,\qqq):=\left\{i \in\I \virg  g_i(t,\qqq)= 0 \right\}
\label{def:I}
\end{equation} 
and for every $\rho >0 $, the following sets:
\begin{equation}
 I_\rho(t,\qqq):=\left\{i \in\I \virg  g_i(t,\qqq) \leq \rho \right\}.
\label{def:Irho}
\end{equation} 
In addition we assume there exist $\gamma >0 $ and $\rho >0 $ such that for all $t \in [0,T] $,
\begin{equation}
 \tag{A5}
\forall \, \qqq \in Q(t) \virg \forall \, \lambda_{i} \geq 0, \sum_{i \in I_\rho(t,\qqq)} \lambda_{i} | \nabla_\qqq \, g_i(t,\qqq)|  \leq  \gamma\left| \sum_{i \in I_\rho(t,\qqq)}  \lambda_{i} \nabla_\qqq \, g_i (t,\qqq) \right|.
\label{inegtrianginverserho}
\end{equation}
We will use the following weaker assumption too:
\begin{equation}
 \tag{A5'}
\forall \, \qqq \in Q(t) \virg \forall \, \lambda_{i} \geq 0, \sum_{i \in I(t,\qqq)} \lambda_{i} | \nabla_\qqq \, g_i(t,\qqq)|  \leq  \gamma\left| \sum_{i \in I(t,\qqq)}  \lambda_{i} \nabla_\qqq \, g_i (t,\qqq) \right|.
\label{inegtrianginverse}
\end{equation}

\noi In particular, this last assumption implies that for all $t $, the gradients of the active inequality constraints $\nabla_\qqq \, g_i (t,\qqq) $ are positive-lineary independent at all $\qqq \in Q(t)$, which is usually called the Mangasarian-Fromowitz constraint qualification (MFCQ). Conversely the MFCQ condition at a point $\qqq $ yields a local version of Inequality (\ref{inegtrianginverse}). \\


\noi The notion of uniform prox-regularity allows to adapt the catching-up algorithm because the projection onto a uniformly prox-regular set is well-defined in its neighbourhood. However, from a numerical point of view, it may be difficult to perform this projection. Here, we study a numerical scheme avoiding this difficulty, which is adapted from the one proposed in~\cite{Maurygrain, esaim, CRAS}. The idea is to replace $Q(t)$ with a convex set $\Qc(t,\qqq)$ (depending on the position).
This substitution is convenient because classical methods can be employed to compute the projection onto a convex set.  Yet this replacement raises some difficulties for the numerical analysis which are solved in proving that for every $ \qqq \in Q(t)$, the set $\Qc(t,\qqq)$ is a good local approximation of $Q(t)$ around $\qqq$.

\bigskip
\noi The paper is structured as follows: In Section~\ref{sec:math}, we describe the mathematical framework to study the differential inclusion (\ref{incldiff}). By justifying that the set-valued map $Q$ is Lipschitzian and takes  uniformly prox-regular values (Propositions~\ref{Qlip} and~\ref{Qprox}), we obtain well-posedness results for (\ref{incldiff}) in Theorem~\ref{theo:wp}. Then in Section~\ref{sec:schema} we propose a prediction-correction scheme (\ref{schema}) and prove its convergence in Theorem~\ref{theo:qhq}. Finally we apply these results in two examples in Section~\ref{sec:appli}. The first situation is a case in point and the second one comes from the modelling of crowd motion in emergency evacuation.

\section{Mathematical framework and well-posedness results}
\label{sec:math}
\subsection{Preliminaries}
Firstly we recall some definitions and properties to
specify the mathematical framework. Here we consider a finite dimension space $\R^d $ equipped with its Euclidean structure although these notions have been extended in a Hilbertian context. For more details, we refer the reader to \cite{Clarke, Bounkhel}.
\begin{dfntn}
Let $S$ be a closed subset of $\R^d $.\\ The proximal
normal cone to $S$ at $\xx$ is defined by:  $$ \NN(S,\xx):=
\left \{
\vv \in \R^d, \  \exists \alpha >0, \  \xx \in \P_S(\xx + \alpha
\vv) \right \}, $$
where  $$\P_S(\yy):=\{\zz \in S, \ d_S(\yy)= |\yy -\zz |\}, \textmd{ with } d_S(\yy):=\inf_{\zz \in S} |\yy- \zz|  $$
corresponds to the Euclidean projection onto $S$.
\label{def:N}
\end{dfntn}
\noi Note that for all $\xx \in int(S):= S \setminus \partial S,$ $\NN(S,\xx)=\{0\}. $ This concept extends in a certain way the notion of normal outward direction for a smooth manifold, as specified in the next proposition. 
\begin{prpstn}
Let $S$ be a closed subset of $\R^d$ whose boundary $\partial S$ is an oriented $C^2$ hypersurface. For each $\xx \in \partial S$, we denote by $\nu(\xx)$ the outward normal to $S$ at $\xx$.
Then, for each $\xx \in \partial S$, the proximal normal cone to $S$ at $\xx$ is generated by $\nu(\xx)$, i.e. $$ \NN(S,\xx) = \R^+ \nu(\xx).$$ 
\label{conepnlisse} 
\end{prpstn}
\noindent We refer the reader to Proposition 3.2 of~\cite{thesejv} for a detailed proof. \\ 
\noindent Using~\cite{Clarke}, we recall the concept of uniform prox-regularity as follows:
\begin{dfntn}
Let $S$ be a closed subset of $\R^d$. $S$ is said $\eta$-prox-regular if for
all $\xx \in \partial S $ and $\vv \in \NN(S,\xx), \  |\vv|=1 $ we have:
$$ B(\xx+\eta \vv, \eta)\cap S = \emptyset.  $$
Equivalently, $S$ is $\eta$-prox-regular if for
all $\yy \in S$, $\xx \in \partial S $ and $\vv \in \NN(S,\xx)$, 
$$ \langle \vv , \yy -\xx \rangle \leq \frac{|\vv|}{2 \eta} |\yy - \xx|^2.$$
\label{caracpr}
\end{dfntn}
\noindent In other words, $S$ is $\eta$-prox-regular if any external ball with radius smaller than $\eta$ can be rolled around it (see Fig~\ref{fig:ensproxreg}).  
Moreover, this definition ensures that the projection onto such a set is well-defined and is continuous in
its neighbourhood. Note that a closed convex set $C \subset \R^d$ is $\infty$-prox-regular.
\begin{figure}
\centering
\psfrag{e}[c]{$\eta$}
\psfrag{x}[c]{$\xx$}
\psfrag{s}[c]{$S$}
\includegraphics[width=0.23\textwidth]{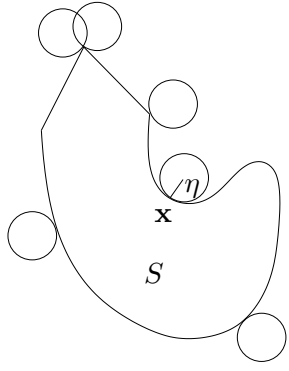}
\caption{$\eta $-prox-regular set.}
\label{fig:ensproxreg} 
\end{figure} 
\noindent We claim a technical lemma which will be useful later.
\begin{lmm}
Let $S$ be a closed convex set included in $\R^d$. Then for all  $\xx \in
S$ and  $\ww \in \R^d$:
 $$
\begin{array}{llll}
\ww \in \NN(S,\xx) &\stackrel{\textmd{def}}{\Leftrightarrow}& \xx=
P_S(\xx+\ww)&(a) \vsp \\
& \Leftrightarrow & \forall \yy \in S, \ \langle \ww,\yy -\xx \rangle \leq 0 &
(b) \vsp \\
&\Leftrightarrow & \forall \bfxi \in \R^d, \ \langle
\ww,\bfxi \rangle \leq |\ww|\ d_S(\bfxi + \xx)& (c) \vsp \\
& \Leftrightarrow &  \exists k >0, \ \exists \delta >0, \  \forall \vv \in \R^d, \ |\vv| < \delta, \ \langle
\ww,\vv \rangle \leq k\ d_S(\vv+ \xx) & (d)
\end{array}
$$
\label{fondamental}
\end{lmm}
\noindent The elementary proof is left to the reader.

\subsection{Uniform prox-regularity of sets $Q(t)$}
This subsection is devoted to specify the proximal normal cone to $Q(t)$ (defined by (\ref{def:Q})) for $ t\in [0,T]$ and to justify the uniform prox-regularity of this set under the assumptions (\ref{gradg}), (\ref{hessg}) and (\ref{inegtrianginverse}).
We point out that the time-dependence plays no role in this study.

\noi First we deal with only one constraint.
Fix $i \in \I$ and consider the set $ Q_{i}(t)$ defined by (\ref{def:Qi}).
The smoothness of the function $g_i(t,\cdot)$ allows us to apply Proposition~\ref{conepnlisse} and to deduce the expression of the proximal normal cone to $Q_{i}(t)$.
\begin{crllr}
For all $t \in [0,T]$ and $\qqq \in \partial Q_{i}(t) $, 
  $$ \NN(Q_{i}(t), \qqq)=-\R^+ \nabla_\qqq \, g_{i}(t,\qqq) . $$
\end{crllr}
\noindent By Definition~\ref{caracpr}, the constant of prox-regularity equals to the largest radius of a ``rolling external ball''. In order to estimate its radius, tools of differential geometry can be used. More precisely, to show that the set $Q_{i}(t)$ is uniformly prox-regular, we can apply the following theorem proved in~\cite{Delgado}.
\begin{thrm}
Let $C$ be a closed convex subset of $\R^d$ such that $\partial C$ is an oriented $C^2$ hypersurface of $\R^d$. We denote by $\nu_C(\xx)$ the outward normal to $C$ at $\xx$ and by $\rho_1(\xx),..,\rho_{d-1}(\xx) \geq 0 $ the principal curvatures of $C$ at $\xx$.
We suppose that $$ \rho := \sup_{\xx \in \partial C} \ \sup_{ 1 \leq i\leq d-1 } \ \rho_i(\xx) < \infty.$$
Then $S = \R^d \setminus int(C)$ is a $\eta$-prox-regular set with $\eta = \dsp \frac{1}{\rho}$. 
\label{Weingarten}
\end{thrm}

\begin{prpstn}
For all $t \in [0,T]$ and $i \in \I$,  $Q_{i}(t)$ is $\eta_0 $-prox-regular, with $\eta_0 = \frac{\alpha}{M}$.
\label{etaij}
\end{prpstn} 
\begin{proof}
Let $t \in [0,T]$, by Assumption (\ref{gradg}), the complement of $Q_{i}(t)$ is obviously the interior of the closed convex set $C=\{ g_i(t,\cdot) \leq 0\}$ which satisfies the assumptions of Theorem~\ref{Weingarten}. The constant of prox-regularity of $Q_{i}(t)$ can be obtained by calculating its principal curvatures which are the eigenvalues of the Weingarten endomorphism. 
Let $\qqq \in  \partial Q_{i}(t)= \partial C $, the outward normal to $C$ at $\qqq$ is equal to $-\nu(\qqq)$, where $$ \nu(\qqq) =-\frac{\nabla_\qqq \, g_{i}(t,\qqq)}{|\nabla_\qqq \, g_{i}(t,\qqq)|}. $$ We can specify the expression of the Weingarten endomorphism $ \WW_\qqq$. 
For every tangent vector $\hh \in T_{\qqq}(\partial Q_{i}(t))$, 
$$ \WW_\qqq(\hh):= - \mathrm{D} \nu(\qqq)[\hh]= \frac{1}{|\nabla_\qqq \, g_i(t,\qqq)|} \P_{(\nu(\qqq))^\perp} \left(\mathrm{D}_\qqq^2 g_i(t,\qqq)[\hh] \right),$$
with $$\dsp \P_{\nu^\perp}(\xx)= \xx - \langle \xx , \nu \rangle \nu .$$
By (\ref{gradg}) and (\ref{hessg}), for all $\qqq \in \partial Q_i(t) \subset U_i(t)$, the eigenvalues of $\WW_\qqq$ are bounded by $M\alpha^{-1}$, which ends the proof according to Theorem~\ref{Weingarten} . 
\end{proof}
\noindent Now let us study the feasible set $Q(t)$ that is the intersection of
 all sets $Q_{i}(t) $. We begin to determine its proximal normal cone.
 
\begin{prpstn}
For all $t \in [0,T] $ and $\qqq \in Q(t) $, $$\NN(Q(t),\qqq )= \sum \NN(Q_{i}(t),\qqq ) = - \sum_{i \in I(t,\qqq)} \R^+ \nabla_\qqq \, g_i(t,\qqq).$$
\label{coneprox}
\end{prpstn}
\noindent We refer the reader to Proposition 2.16 in~\cite{M2AN} for a detailed proof. We emphasize that the inclusion $\sum \NN(Q_{i}(t),\qqq ) \subset \NN(Q(t),\qqq )$ always holds. However, the given proof of the other inclusion requires the convexity of functions $g_i(t,\cdot) $ and the boundedness of gradients $\nabla_\qqq \, g_i(t,\cdot)$ (Assumption~(\ref{gradg})). \\

\noindent We now come to the main result of this subsection: the set-valued map $Q$ takes uniformly prox-regular values (with a time-independent constant) which rests on the inverse triangle inequality (\ref{inegtrianginverse}).

\begin{prpstn}
For every $t \in [0,T] $, $Q(t)$ is $\eta$-prox-regular with $$ \eta = \dsp \frac{\eta_0}{\gamma}=  \dsp \frac{\alpha}{M \gamma} . $$ 
\label{Qprox}
\end{prpstn} 
\begin{proof}
Let $t \in [0,T]$ and consider $\qqq \in \partial Q(t)$ and $
\vv \in \NN(Q(t),\qqq)\setminus \{0\}$.
By Proposition~\ref{coneprox}, there exist $\alpha_i \geq 0 $ such that $$\vv = -\sum_{i \in I(t,\qqq)}
\alpha_{i} \nabla_\qqq \, g_{i}(t,\qqq).$$ 
By definition (\ref{def:I}) of $I(t,\qqq) $, $\qqq \in \partial Q_i(t)$ for every $i \in I(t,\qqq)$.
By Proposition~\ref{etaij}, for all $i \in I(t,\qqq)$, $ - \alpha_i \nabla_\qqq \, g_i(t,\qqq) \in \NN(Q_{i}(t),\qqq)$, so we have by Definition~\ref{caracpr},
\begin{equation*}
\langle  - \alpha_i \nabla_\qqq \, g_i(t,\qqq), \tilde{\qqq}-\qqq \rangle  \leq \frac{| \alpha_i \nabla_\qqq \, g_i(t,\qqq)|}{2
  \eta_{0}}|\tilde{\qqq}-\qqq|^2, \ \forall \tilde{\qqq} \in Q_{i}(t) .
\end{equation*}
Since $Q(t) \subset Q_{i}(t)$, by summing these inequalities for $ i \in I(t,\qqq)$, we obtain
\begin{equation*}
\left\langle -\sum_{i \in I(t,\qqq)} \alpha_{i}
\nabla_\qqq \, g_{i}(t,\qqq),  \tilde{\qqq}-\qqq \right\rangle \leq \sum_{i \in I(t,\qqq)}\frac{ \alpha_{i}
  |\nabla_\qqq \, g_{i}(t,\qqq)|}{2
  \eta_{0}}|\tilde{\qqq}-\qqq|^2, \ \forall \tilde{\qqq} \in Q(t).
\end{equation*}
>From (\ref{inegtrianginverse}), it follows that 
\begin{equation*}
\left \langle -\sum_{i \in I(t,\qqq)} \alpha_{i}
\nabla_\qqq \, g_{i}(t,\qqq), \tilde{\qqq}-\qqq \right \rangle \leq  \frac{\gamma}{2 \eta_0}\left|\sum_{i \in I(t,\qqq)} \alpha_{i}
  \nabla_\qqq \, g_{i}(t,\qqq)\right| |\tilde{\qqq}-\qqq|^2, \ \forall \tilde{\qqq} \in Q(t).
\end{equation*}
We deduce from Definition~\ref{caracpr} that $Q(t)$ is $\eta$-prox-regular with $\eta= \frac{\eta_0}{\gamma}.$
\end{proof}
\subsection{Lipschitz regularity of $Q$}
Here we check that $t \to Q(t) $ varies in a Lipschitz way. Aiming that, we need this technical lemma:
\begin{lmm}
 There exists $ \delta >0$ such that for all $t \in [0,T] $ and $ \qqq \in Q(t)$:
 \begin{equation} \exists \uu \in \R^d,\  |\uu|=1,\ \forall i\in I_\rho(t,\qqq) \virg \langle \nabla_\qqq \, g_i(t,\qqq),\uu\rangle \geq \delta, \label{eq:gooddir}\end{equation}
\label{gooddir}
\end{lmm}
\begin{proof}
 We set the following cone
$$ \NN_\rho(Q(t),\qqq) := -\sum_{i\in I_\rho(t,\qqq)} \R^+ \nabla_\qqq \, g_i(t,\qqq)$$
and its polar cone
$$ C_\rho(Q(t),\qqq) := \NN_\rho(Q(t),\qqq)^\circ:=\{ \ww \in \R^d \virg \forall \vv \in \NN_\rho(Q(t),\qqq) \virg \ll  \vv, \ww\rr \leq 0\} .$$
According to the classical orthogonal decomposition of a Hilbert space as the sum of mutually polar cones (see~\cite{Moreaucones}), we have:
$$ Id= \P_{\NN_\rho(Q(t),\qqq)} + \P_{C_\rho(Q(t),\qqq)},$$
where $\P$ denotes the Euclidean projection. \\
So let us consider for $i\in I_\rho(t,\qqq)$ the corresponding decomposition of $\nabla_\qqq \, g_i(t,\qqq)$:
$$ \nabla_\qqq \, g_i(t,\qqq) = a_i+b_i \in \NN_\rho(Q(t),\qqq) + C_\rho(Q(t),\qqq). $$
Assumption (\ref{gradg}) gives us: $|b_i|\leq |\nabla_\qqq \, g_i(t,\qqq)| \leq \beta$.
Since $a_i \in \NN_\rho(Q(t),\qqq) $, it can be written: $a_i = - \sum \lambda_j  \nabla_\qqq \, g_j(t,\qqq) \virg \lambda_j \geq 0$ involving $$|b_i|= |\nabla_\qqq \, g_i(t,\qqq)-a_i|= \left| \sum_{j \neq i}\lambda_j  \nabla_\qqq \, g_j(t,\qqq)+(1+ \lambda_i) \nabla_\qqq \, g_i(t,\qqq) \right|  .$$
Then using the inverse triangle inequality (\ref{inegtrianginverserho}) and Assumption (\ref{gradg}), we get~:
$$ |b_i| \geq \frac{\alpha}{\gamma} \left(\sum \lambda_j +1 \right) \geq \frac{\alpha}{\gamma}.$$
As a consequence, it comes: 
\begin{equation}
 \frac{\alpha}{\gamma} \leq |b_i| \leq \beta.
\label{bi}
\end{equation}
Since $0\in \NN_\rho(Q(t),\qqq)$ and $a_i=\P_{\NN_\rho(Q(t),\qqq)}(\nabla_\qqq \, g_i(t,\qqq))$, we obtain
\begin{align} 2\langle b_i,-\nabla_\qqq \, g_i(t,\qqq) \rangle & =  |b_i-\nabla_\qqq \, g_i(t,\qqq)|^2-|b_i|^2 - |\nabla_\qqq \, g_i(t,\qqq)|^2 \nonumber \\
& =|a_i|^2 -|b_i|^2 - |\nabla_\qqq \, g_i(t,\qqq)|^2 \nonumber \\ 
& \leq -|b_i|^2 \leq -\frac{\alpha^2}{\gamma^2}. 
\label{eq:gooddir2} 
\end{align}
Now we set 
\begin{equation}
\uu:=\frac{\sum_{i\in I_\rho(t,\qqq)} b_i}{|\sum_{i\in I_\rho(t,\qqq)} b_i|} \in C_\rho(Q(t),\qqq).
\label{defu}
\end{equation}
This is well-defined because (\ref{eq:gooddir2}) and Assumption (\ref{gradg}) imply that for any $j\in I_\rho(t,\qqq)$
\begin{equation}\left|\sum_{i\in I_\rho(t,\qqq)} b_i\right| \geq \frac{1}{\beta}\left\langle \sum_{i\in I_\rho(t,\qqq)} b_i,\nabla_\qqq \, g_j(t,\qqq) \right\rangle \geq \frac{1}{\beta}\left\langle b_j, \nabla_\qqq \, g_j(t,\qqq)      \right\rangle            \geq \frac{\alpha^2}{2\beta \gamma^2}.
\label{sombi}
\end{equation}
\noi Then (\ref{eq:gooddir}) follows from (\ref{bi}) and (\ref{sombi}) with $$ \dsp \delta=\frac{\alpha^2}{2 \gamma^2 p\beta}.$$
\end{proof}

\begin{prpstn} \label{Qlip} The set-valued map $Q$ is Lipschitz continuous with respect to the Hausdorff distance. More precisely there exists $K_L >0 $ such that $$ \forall t,s \in [0,T]  \virg d_H(Q(t),Q(s)) \leq  K_L |t-s |.$$
\end{prpstn}
\begin{proof}
Consider $t,s \in [0,T]$ and $\qqq\in Q(t)$, let us construct a point closed to $\qqq$ belonging to $Q(s)$.
Let $\uu $ given by Lemma \ref{gooddir}, we introduce 
 $\zz(h):=\qqq+h\uu$ with $h >0$. We claim that for $h< h_l:=\min(c,\rho/(\beta +\delta)) $ , 
$$
\dsp  \forall i\in \I, \qquad g_i(t,\zz(h))\geq h \delta$$ 
(where $c $ and $\delta $ are introduced in (\ref{Ui}) and (\ref{gooddir})).
Indeed for  $h< c$, due to the convexity of $ g_i(t,\cdot)$, it comes
$$
g_i(t,\qqq+h\uu) \geq g_i(t,\qqq) + h \ll \nabla_\qqq \, g_i(t,\qqq), \ \uu \rr.
$$
 As a consequence, for $i\in I_\rho(t,\qqq)$,
$$ \dsp g_i(t,\qqq+h\uu) \geq h\delta ,$$
by (\ref{eq:gooddir}).
Moreover for every $i \notin  I_\rho(t,\qqq) $, according to Assumption (\ref{gradg}), $$g_i(t,\qqq+h\uu) \geq \rho -h\beta \geq h\delta $$  if $\dsp h<\frac{\rho}{\beta + \delta} $.
Thus for $h<h_l$, we have $g_i(t,\qqq+h\uu)\geq h \delta$ for all $i \in \I$. 
That is why we deduce from Assumption (\ref{dtg}) that $\zz(h)\in Q(s)$ if $ h \delta \geq \beta |t-s|$. Setting $\ell:=\dsp  \frac{\beta  h_l}{\delta} $, if $|t-s|<\ell $, it can be written that
$$ d_{Q(s)}(\qqq) \leq \inf_{h \delta  \geq  \beta|t-s|} |\qqq -\zz(h)| \leq \frac{ \beta}{\delta} |t-s|.$$
Consequently we obtain if $|t-s|<\ell $,
$$ d_H(Q(t),Q(s)) = \max\left(\sup_{\qqq\in Q(t)} d_{Q(s)}(\qqq), \sup_{\qqq\in Q(s)} d_{Q(t)}(\qqq) \right)\leq  \frac{\beta}{\delta} |t-s|.$$
This inequality is actually satisfied for any $t, s \in [0,T]$. To check it, it suffices to divide the corresponding interval into subintervals 
of length $\ell$ and to apply the triangle inequality. So the required result is obtained with $K_L=\beta/\delta $.

\end{proof}

\subsection{Well-posedness results}

\noi We now come to the main result. 

\begin{thrm}
\label{theo:wp}
Let $T>0 $ and $ \ff :[0,T] \times \R^d \rightarrow \R^d $ be a measurable map satisfying:
\begin{align}
& \forall \qqq \in \bigcup_{s \in [0,T]} Q(s) \virg  \ff( \cdot,\qqq) \textmd{ is Riemann-integrable on } [0,T]  \label{riemann} \vsp \\
 & \exists K >0 \virg \forall \qqq \in \bigcup_{s \in [0,T]} Q(s) \virg \forall t \in [0,T] \virg  | \ff(t,\qqq) -\ff(t,\tilde{\qqq})| \leq K |\qqq - \tilde{\qqq}|
\label{lip} \vsp \\
& \exists L >0 \virg \forall \qqq \in \bigcup_{s \in [0,T]} Q(s) \virg \forall t \in [0,T] \virg | \ff(t,\qqq)|\leq L(1+|\qqq|).
\label{lingro}
\end{align}
Then, under Assumption (\ref{inegtrianginverserho}) for all $\qqq_0 \in Q(0) $, the
following problem 
\begin{equation}
 \left \{ 
\begin{array}{l}
\displaystyle\frac {d \qqq}{dt}(t) + \NN(Q(t),\qqq(t)) \ni \ff(t,\qqq(t)) \ \textmd{for a.e. } t \in [0,T]  \vsp
\\
\qqq(0)=\qqq_0,
\end{array}
\right.
\label{incldiff2}
\end{equation}
has one and only one absolutely continuous solution $\qqq $ satisfying $ \qqq(t) \in Q(t)$ for every $t \in [0,T] $.
\end{thrm}

\begin{proof}
As the set-valued map $Q $ varies in a Lipschiz way by Proposition~\ref{Qlip} and for all $t \in [0,T] $, $Q(t)$ is $\eta$-prox-regular by Proposition~\ref{Qprox}, we can apply Theorem 1 of~\cite{Thibrelax}.  
\end{proof}

\noi We now consider a constant set-valued map $Q$ (the constraints $g_i $ are supposed to be time-independent).
In this particular case, we only require Assumption (\ref{inegtrianginverse}) instead of Assumption (\ref{inegtrianginverserho}).

\begin{thrm}
\label{theo:wp2}
Under Assumptions (\ref{Ui}), (\ref{gradg}), (\ref{hessg}), (\ref{inegtrianginverse}) with time-independent set $Q$, 
let $T>0 $ and $ \ff :[0,T] \times Q \rightarrow \R^d $ be a measurable map satisfying:
\begin{align}
& \forall \qqq \in Q \virg  \ff( \cdot,\qqq) \textmd{ is Riemann-integrable on } [0,T]  \label{riemannbis} \vsp \\
 & \exists K >0 \virg \forall \qqq \in Q \virg \forall t \in [0,T] \virg  | \ff(t,\qqq) -\ff(t,\tilde{\qqq})| \leq K |\qqq - \tilde{\qqq}|
\label{lipbis} \vsp \\
& \exists L >0 \virg \forall \qqq \in Q \virg \forall t \in [0,T] \virg | \ff(t,\qqq)|\leq L(1+|\qqq|).
\label{lingrobis}
\end{align}
Then, for all $\qqq_0 \in Q $, the
following problem 
\begin{equation}
 \left \{ 
\begin{array}{l}
\displaystyle\frac {d \qqq}{dt}(t) + \NN(Q,\qqq(t)) \ni \ff(t,\qqq(t)) \ \textmd{for a.e. } t \in [0,T]  \vsp
\\
\qqq(0)=\qqq_0,
\end{array}
\right.
\label{incldiff2bis}
\end{equation}
has one and only one absolutely continuous solution $\qqq $ taking values in $Q$.  
Moreover the map $\qqq $ is also solution of the following differential equation:
\begin{equation}
 \left \{ 
\begin{array}{l}
\displaystyle\frac {d \qqq}{dt}(t) +\P_{\NN(Q,\qqq(t))}(\ff(t,\qqq(t)))= \ff(t,\qqq(t)) \ \textmd{for a.e. } t \in [0,T]  \vsp
\\
\qqq(0)=\qqq_0,
\end{array}
\right.
\label{equadiff2}
\end{equation}
\end{thrm}

\noi For this last point we refer the reader to Proposition 3.3 in \cite{Fred}.

\begin{rmrk}
 The assumptions about the perturbation $\ff$ can be weakened (see~\cite{Thibrelax}). However they are required in the numerical analysis of the scheme proposed in the following section. More precisely, we need Assumption (\ref{riemann}) to prove Lemma~\ref{convfaiblefn}. 
\end{rmrk}

\section{A numerical scheme}
\label{sec:schema}
\subsection{Presentation}
We present in this section a numerical scheme to approximate the solution
 of~(\ref{incldiff2}) on the time interval $[0, T] $. This scheme is adapted from the one proposed by B. Maury for granular media in~\cite{Maurygrain}. Let $ n\in \mathbb{N}^\star$, $h = T/n $ be the time
 step and $t_k^n=k h $ be
 the computational times. 
 We denote by 
$\qqq_k^n $ the approximation of  $\qqq(t_k^n )$ with $\qqq_0^n=\qqq_0 $. The next configuration is computed as follows:
\begin{equation}
 \qqq_{k+1}^n : = \P_{\Qc( t_{k+1}^n ,\qqq_k^n)} (\qqq_k^n  + h \ff(t_k^n, \qqq_k^n ))
\label{schema}
\end{equation} 
with
$$
\Qc(t,\qqq)  := \{ \tilde \qqq \in \R^d \virg g_{i}(t,\qqq) + \langle \nabla_\qqq \, g_{i}(t,\qqq), 
\tilde \qqq - \qqq \rangle \geq 0 \quad \forall \, i 
\} \textmd{ for } \qqq \in U(t):=\bigcap U_i(t).
$$
We recall that all the gradients $\nabla_\qqq \, g_{i}(t,\qqq)$ are well-defined provided that $\qqq\in U(t)$.
The set $\Qc(t,\qqq)  $ can be seen as an inner convex approximation of $Q(t)$ with respect
to $\qqq$. 
This scheme is a prediction-correction algorithm: predicted position vector $\qqq_k^n  + h \ff(t_k^n,\qqq_k^n )$, that may not be admissible, is projected onto the approximate set of feasible configurations.

\noi Let us check that this scheme is well-defined for $h < \frac{c}{K_L} $ with $c$ and $K_L $ respectively given by Assumption~(\ref{Ui}) and Proposition~(\ref{Qlip}):
\begin{prpstn} \label{schemewd}
Assume that $ h K_L<c $, then for all $k < n-1$
\begin{equation*}
 \qqq_{k+1}^n \in \Qc(t_{k+1}^n, \qqq_k^n ) \subset Q(t_{k+1}^n) \subset U(t_{k+2}^n) .
\end{equation*}
Thus every computed configuration is feasible .
\end{prpstn}
\begin{proof}
 Since $\qqq_0^n=\qqq_0 \in Q(0) $ and $d_H(Q(0), Q(t_1^n)) \leq K_L h $, we get $$ \qqq_0^n \in Q(t_1^n) + h K_L \overline{B(0,1)} \subset U(t_1^n), $$ according to (\ref{Ui}) and Proposition~\ref{Qlip}.
So $\Qc(t_{1}^n, \qqq_{0}^n )$ is well-defined and is included in $Q(t_{1}^n) $ due to the convexity of the functions $g_i(t_{1}^n, \cdot) $.
Then as $\qqq_{1}^n \in \Qc(t_{1}^n, \qqq_{0}^n )$, $\qqq_{1}^n \in Q(t_{1}^n) \subset Q(t_2^n) + h K_L \overline{B(0,1)} \subset U(t_2^n)$. 
By iterating, we end the proof.
\end{proof}

\noindent By Lemma~\ref{fondamental} (a), it can be checked that
\begin{equation}
\label{eq:diffincdisc}
\frac {\qqq_{k+1}^n - \qqq_k^n }{h} + \NN(\Qc(t_{k+1}^n, \qqq_k^n ),\qqq_{k+1}^n ) \ni
\ff(t_k^n,\qqq_k^n ),
\end{equation}
so that the scheme can also be seen as a 
semi-implicit discretization  of~(\ref{incldiff2}),
 where the cone $ \NN(\Qc(t_{k+1}^n, \qqq_k^n ),\qqq_{k+1}^n ) $ stands for an approximation of $ \NN(Q(t),\qqq(t) )$.
 \begin{center}
\begin{figure}
\centering
\begin{tabular}{c}
\psfrag{a}[l]{$\qqq_k^n$}
\psfrag{b}[l]{$\qqq_k^n +h \ff(t_k^n, \qqq_k^n) $}
\psfrag{c}[l]{$\qqq_k^n +h \ff(t_k^n, \qqq_k^n)$}
\psfrag{d}[l]{$\qqq_{k+1}^n$}
\psfrag{e}[l]{$\qqq_{k+1}^n$}
\psfrag{f}[l]{$\hspace{-5.5pt}\tilde{\qqq}_{k+1}^n$}
\psfrag{g}[l]{$ \tilde{\qqq}_{k+1}^n$}
\psfrag{q}[l]{$ \LARGE{Q(t_{k+1}^n)} $}
\psfrag{t}[l]{\hspace{-1cm}$ \Large{\Qc(t_{k+1}^n,\qqq_k^n)}$}
\includegraphics[scale=0.48]{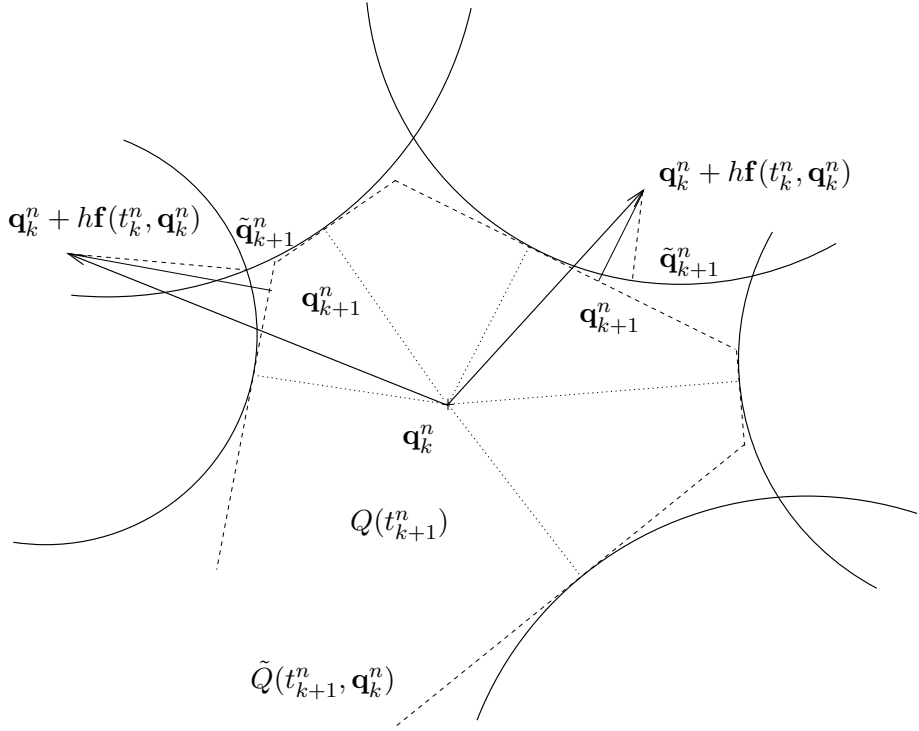} 
\end{tabular}
\caption{Theoretical and numerical projections.}
\label{fig:projs}
\end{figure}
\end{center} 
 \noindent In Figure~\ref{fig:projs},  we illustrate the set $Q(t_{k+1}^n) \subset \R^d$, intersection of sets $Q_i(t_{k+1}^n)$ whose boundaries are plotted in solid line. The set $\Qc(t_{k+1}^n,\qqq_k^n)$ is delimited by the dashed line. The theoretical and numerical projections, respectively $\tilde{\qqq}_{k+1}^n := \P_{Q(t_{k+1}^n)} (\qqq_k^n + h \ff(t_k^n,\qqq_k^n))$ and $\qqq_{k+1}^n $ are represented (for two examples of $\ff(t_k^n,\qqq_k^n)$). Indeed, as $Q(t_{k+1}^n) $ is uniformly prox-regular, the projection onto $Q(t_{k+1}^n)$ of $\qqq_k^n + h \ff(t_k^n,\qqq_k^n) $ is well-defined for $h$ small enough. 
 The replacement of $Q(t_{k+1}^n)$ by the convex set $\Qc(t_{k+1}^n,\qqq_k^n)$ is convenient because it allows us to use classical numerical methods to compute this projection. However, it raises some difficulties to prove the convergence of the scheme. On the one hand, we have to check that this approximation is sufficiently accurate (see Lemma~\ref{cones_normaux}). On the other hand, the set $\Qc(t,\qqq)$ does not vary smoothly enough to directly apply results about sweeping process (this point will be specified in Remark~\ref{localprop}).
\begin{lmm}
For all $t \in [0,T] $ and $\qqq \in Q(t) $, $$\NN(Q(t),\qqq )= \NN(\Qc(t,\qqq),\qqq ) .$$
\label{cones_normaux}
\end{lmm} 
\begin{proof}
 We need the well-known following result (see e.g.~\cite{Ciarlet}) which rests on the reformulation of the minimization problem under constraints (associated to the projection), in terms of a saddle-point problem.
\begin{lmm}
Let $t \in [0,T] $, $\qqq \in Q(t)$ and $\tilde{\qqq} \in \R^{d} $. Then,
$\tilde{\pp}=\textrm{P}_{\Qc(t,\qqq)} (\tilde{\qqq}) $ is equivalent to the existence of  $
\bflambda \in (\mathbb{R}^+)^{p}$ such that 
$(\tilde{\pp},\bflambda ) $ satisfies: 
\begin{equation}
\tag{$P_{t,q,\tilde{q}}$} \left\{
\begin{array}{l}
 \dsp \tilde{\pp}= \tilde{\qqq} + \sum \lambda_{i} \nabla_\qqq \, g_{i}(t,\qqq),
\vsp \\
\dsp \forall i \virg g_{i}(t,\qqq)+ \langle \nabla_\qqq \, g_{i}(t,\qqq) ,
\tilde{\pp}-\qqq \rangle \geq 0,
 \vsp \\
\dsp \sum \lambda_{i}\  ( g_{i}(t,\qqq)+ \langle \nabla_\qqq \, g_{i}(t,\qqq) ,
\tilde{\pp}-\qqq  \rangle )=0.
\end{array}
\right.
\label{pqqt}
\end{equation} 
\label{point-selle} 
\end{lmm}
\noindent First we check that $\NN(Q(t),\qqq)\subset
\NN(\Qc(t,\qqq),\qqq)$. \\
Let $\ww \in  \NN(Q(t),\qqq)$. According to Proposition~\ref{coneprox}, this vector can be written $$ \ww = -\sum_{i \in I(t,\qqq)} \mu_{i} \nabla_\qqq \, g_i(t,\qqq)$$ with nonnegative reals
$\mu_{i} $.
So by definition of $\Qc(t,\qqq) $, we get
for all $\pp \in \Qc(t,\qqq)$,  $ \langle\ww,\pp-\qqq  \rangle\leq 0 $. Applying Lemma~\ref{fondamental} (b), we deduce that $\ww \in \NN(\Qc(t,\qqq),\qqq) $. \\
It remains to prove $\NN(\Qc(t,\qqq),\qqq)\subset
\NN(Q(t),\qqq) $. \\
Letting $\ww \in  \NN(\Qc(t,\qqq),\qqq) $, we can write by Lemma~\ref{fondamental} (a), $\qqq=P_{\Qc(t,\qqq)}(\qqq+\ww) $ which implies by Lemma~\ref{point-selle}, $$\qqq =(\qqq +\ww) + \sum \mu_{i}\nabla_\qqq \, g_i(t,\qqq),
\textmd{ where } (\qqq, \bfmu) \in Q(t) \times (\R^+)^p \textmd{ satisfies System }
(P_{t,q,q+w}). $$ Consequently, $\ww =  -\sum \mu_{i} \nabla_\qqq \, g_i(t,\qqq) \virg \mu_{i} \geq 0$. The third relation in ($P_{t,q,q+w}$) is $\sum \mu_{i}g_i(t,\qqq)=0 $, so $\mu_{i} =0$ if $g_i(t,\qqq)>0$, which concludes the proof.
\end{proof}

\subsection{Convergence}
Before stating the result of convergence, we introduce some notations.
We define the piecewise constant function $ \ff^n $ as follows, $$ \ff^n(t)=\ff(t_k^n, \qqq_k^n)
\textmd{ if } t \in
[t_k^n,t_{k+1}^n [,\ k<n \textmd{ and } \dsp \ff^n(T)=\ff(t_{n-1}^n, \qqq_{n-1}^n).$$ We denote by $\qqq^n$ the continuous, piecewise linear function satisfying for $k \in\{ 0, \dots, n\}$, $\qqq^n(t_k^n)= \qqq_k^n .$
To finish, we introduce the functions $\rho $ and $\theta $ defined by $$\rho^n(t)=t_k^n \textmd{ and }\theta^n(t)=t_{k+1}^n
\textmd{ if } t \in [t_k^n, t_{k+1}^n[,\  \rho^n(T)=T \textmd{ and }\theta^n(T)=T. $$
\begin{thrm}
\label{theo:qhq}
 With the assumptions of Theorem~\ref{theo:wp}, $\qqq^n$ tends to $\qqq$
 in $C^0([0,T], \R^d)$, where $t\mapsto \qqq(t)$ is the unique solution
of~(\ref{incldiff2}).   
\end{thrm}
\begin{proof}
First we prove the boundedness of the sequence $(\qqq^n) $ and to do this, we use the same arguments as developed in~\cite{Thibrelax} with the following lemma.

\begin{lmm}
There exists $D>0 $ such that for $h $ small enough,  
$$d_{\Qc(t_{k+1}^n, \qqq_k^n)}(\qqq_k^n ) \leq D h.$$
 \label{distQc}
\end{lmm}

\noi Without loss of generality, we can assume that $T<1/(2D+4L)$ and $h$ small enough to apply Lemma~\ref{distQc}. \\
It follows from (\ref{lingro}) that 
\begin{align}
 | \qqq_{k+1}^n - (\qqq_k^n+ h \ff(t_k^n,\qqq_k^n)) |& = d_{\Qc(t_{k+1}^n, \qqq_k^n)}(\qqq_k^n+ h \ff(t_k^n,\qqq_k^n) ) \nonumber \\
& \leq d_{\Qc(t_{k+1}^n, \qqq_k^n)}(\qqq_k^n )+ h |  \ff(t_k^n,\qqq_k^n)| \nonumber \\
& \leq  h (D+L) (1+ |\qqq_k^n|).
\label{majproj}
\end{align}
We deduce that 
\begin{equation}
| \qqq_{k+1}^n - \qqq_k^n| \leq h (D+2L) (1+ |\qqq_k^n|)
 \label{difqk}
\end{equation}
which implies that 
$$|\qqq_{k+1}^n| \leq |\qqq_k^n|+ h (D+2L) (1 +|\qqq_k^n|)  .$$
Hence $$ |\qqq_{n}^n| \leq |\qqq_0|+ h (D+2L) \sum_{k=0}^{n-1} (1 +|\qqq_k^n|) \leq |\qqq_0| + nh (D+2L)  (1 +\max_{0\leq k\leq n}|\qqq_k^n|). $$
Since $T=nh $, we deduce that $$\max_{0\leq k\leq n}|\qqq_k^n| \leq \frac{|\qqq_0| +T(D+2L)}{1-T(D+2L)} \leq 2 (|\qqq_0| +T(D+2L)).$$ As a consequence, we obtain by letting $C:=2 (|\qqq_0| +T(D+2L)) $
 \begin{equation}
 \| \qqq^n \|_\infty \leq C.
\label{majqn}
 \end{equation}
By (\ref{difqk}) and (\ref{majqn}), we have 
 \begin{equation}
 \left \| \frac{d\qqq^n}{dt} \right \|_\infty \leq (D+2L)(1+C).
 \end{equation}

\noi By applying Arzela-Ascoli's Theorem, it can be easily shown that there exists a function $\qqq \in W^{1, \infty}([0,T],Q)$ and a subsequence (still denoted by $\qqq^n$) satisfying $$
\left\{
\begin{array}{l}
 \dsp\frac{d\qqq^n}{dt} \stackrel{\star}{\rightharpoonup} \frac{d\qqq}{dt}
\textmd{ in } L^{\infty} ([0,T],\R^{d})
\vsp \\   \dsp\qqq^n \xrightarrow[n \to \infty ]{} \qqq \textmd{ uniformly in
} [0,T]. 
\end{array}
\right.
$$ 
By (\ref{majqn}), it can be specified that 
\begin{equation}
 \| \qqq \|_\infty \leq C.
\label{majq}
\end{equation}
Since the time-interval is bounded, we get
$$
\dsp\frac{d\qqq^n}{dt} \rightharpoonup \frac{d\qqq}{dt}
\textmd{ in } L^{1} ([0,T],\R^{d}).
$$
Now let us check that the limit function $\qqq $ satisfies the differential inclusion (\ref{incldiff2}).
The beginning of the proof uses classical tools developped e.g. in~\cite{Thibbv}.
We want to show that
$$ \dsp \frac {d\qqq}{dt}(t) - \ff\left(t, \qqq(t) \right) \in -\NN(Q(t),\qqq(t))
\ \textmd{ for a.e. } t \in [0,T]$$ which is equivalent to $$\dsp \frac {d\qqq}{dt}(t) - \ff\left(t,\qqq(t) \right) \in -\NN(\Qc(t,\qqq(t)),\qqq(t))
 \ \textmd{ for a.e. } t \in [0,T] $$
by Lemma~\ref{cones_normaux}.

\begin{lmm}
 $$ \ff^n \rightharpoonup \ff(\cdot,\qqq(\cdot)) \textmd{ in } L^{1} ([0,T],\R^{d}). $$
\label{convfaiblefn}
\end{lmm}
\noi This lemma (later proved) implies that
$$
 \dsp \frac{d\qqq^n}{dt}-\ff^n \rightharpoonup \frac{d\qqq}{dt}- \ff(\cdot,\qqq(\cdot))
\textmd{ in } L^{1} ([0,T],\R^{d}).
$$ 
Consequently, by Mazur's Lemma, there exists a sequence $\zz^n \in   L^{1} ([0,T],\R^{d})$ satisfying
\begin{equation}
 \dsp \zz^n \in \textmd{Conv} \left( \dsp \frac{d\qqq^k}{dt}-\ff^k, \ k \geq n \right)
\label{defzn}
\end{equation}
and
$$\zz^n \xrightarrow[n \to \infty ]{} \frac{d\qqq}{dt}- \ff(\cdot,\qqq(\cdot))
\textmd{ in } L^{1} ([0,T],\R^{d}). $$
Extracting a subsequence, we may suppose that 
\begin{equation}
\zz^n \xrightarrow[n \to
  \infty ]{} \zz= \dsp \frac{d\qqq}{dt}- \ff(\cdot,\qqq(\cdot)) \ \textmd{a.e. in }[0,T].
\label{limzn}
\end{equation}
Furthermore, Inclusion~(\ref{eq:diffincdisc}) can be rewritten for almost every $t \in [0,T]$,
 \begin{equation}\frac{d\qqq^n}{dt}(t)-\ff^n(t) \in -
\NN(\Qc(\theta^n(t),\qqq^n(\rho^n(t))),\qqq^n(\theta^n(t))). 
\label{incl}
\end{equation}
Let $t  \in [0,T] $ such that $\zz_n(t) $ tends to $\zz(t)$ and the above differential inclusion holds.
By~(\ref{defzn}), it yields  $$ 
 \forall \bfxi \in \R^d, \ \langle
\zz^n(t),\bfxi \rangle \leq \sup_{k\geq n} \ \left\langle
\frac{d\qqq^k}{dt}(t)-\ff^k(t)    ,\  \bfxi\right \rangle.$$
Passing to the limit, we obtain 
 \begin{equation}
 \forall \bfxi \in \R^{d}, \ \langle
\zz(t),\bfxi \rangle \leq \limsup_{n} \ \left\langle
\frac{d\qqq^n}{dt}(t)-\ff^n(t)    ,\  \bfxi\right \rangle.
\label{eq:limsup}
\end{equation}
From~(\ref{incl}) and from Lemma~\ref{fondamental} (c), we get
%
\begin{equation*}
 \forall n, \  \forall\bfxi \in \R^{d},\  \left\langle
\frac{d\qqq^n}{dt}(t)-\ff^n(t)    ,\  \bfxi\right \rangle \leq
\left |\frac{d\qqq^n}{dt}(t)-\ff^n(t)\right | d_{\Qc(\theta^n(t),\qqq^n(\rho^n(t)))}(
 \qqq^n(\theta^n(t)) -\bfxi) .
\end{equation*}
Hence, by (\ref{majproj}) and (\ref{majqn}), this inequality can be specified
\begin{equation}
\forall n, \  \forall\bfxi \in \R^{d},\  \left\langle
\frac{d\qqq^n}{dt}(t)-\ff^n(t),\  \bfxi\right \rangle  \leq
 (D+L)(1+C) \  d_{\Qc(\theta^n(t),\qqq^n(\rho^n(t)))}(
 \qqq^n(\theta^n(t)) -\bfxi).
\label{eq:inegn}
\end{equation}

\noindent Proposition~\ref{lem:proj} with  $\qqq=\qqq(t) $, $\qqq_n=\qqq^n(\rho^n(t)) $ and $t_n=\theta^n(t)$ implies
there exists $\nu >0 $ such that for all $\bfxi \in \R^{d}, \ |\bfxi| < \nu $,
\begin{equation}
\left |d_{\Qc(\theta^n(t),\qqq^n(\rho^n(t)))}(\qqq(t) - \bfxi )- d_{\Qc(t,\qqq(t))}(
\qqq(t) -\bfxi )\right |\xrightarrow[n \to \infty ]{} 0. 
\label{convdist}
\end{equation}
 Finally, by passing to the limit in~(\ref{eq:inegn}), it comes from (\ref{eq:limsup}):
$$ \forall \bfxi \in \R^{d},\ |\bfxi| < \nu , \ \langle
\zz(t),\bfxi \rangle \leq (D+L)(1+C)   d_{\Qc(t,\qqq(t))}(\qqq(t) -\bfxi), $$ which is equivalent to 
\begin{equation} 
 \zz(t) \in -\NN(\Qc(t,\qqq(t)),\qqq(t )),
\label{enfin} 
\end{equation}
by Lemma~\ref{fondamental}
(d). 
The required result follows from (\ref{limzn}) provided we prove Lemmas~\ref{distQc}, \ref{convfaiblefn} and Proposition~\ref{lem:proj}.
\end{proof}
\begin{rmrk}
 To estimate the upper limit in~(\ref{eq:limsup}), we give a different method than the one written in~\cite{Thibbv, Thibrelax}. However the main argument is the same. Indeed in~\cite{Thibrelax}, moving sets $C(t)$ (only depending on time) are considered and the assumed regularity of the set-valued map $C$ involves  $$ d_{C(t^n)} (y) \rightarrow d_{C(t)} (y), \forall y, \textmd{ when } t^n \to t $$ which looks like (\ref{convdist}). This property allows the authors to prove Proposition 2.1 of~\cite{Thibbv}, which implies (\ref{enfin}) from (\ref{eq:limsup}).
In our case, we just need (and only have) a local continuity. More precisely, letting
$$ \Omega:=\left\{(t,\qqq),\ t\in[0,T],\ \qqq\in U(t) \right\}$$
and
$$ \Omega_c:=\left\{(t,\qqq),\ t\in[0,T],\ \qqq\in Q(t) \right\},$$
the map 
$$\begin{array}{cccc}
  \phi: &\Omega \times \R^{d} &\longmapsto &\R  \vsp \\
&(t,\qqq, \tilde{\qqq}) & \longrightarrow &d_{\Qc(t,\qqq)}(\tilde{\qqq})
\end{array}$$
is continuous on a neighbourhood of the set $\{(t,\qqq,\tilde{\qqq}), \ \qqq=\tilde{\qqq}\}$ in $\Omega_c \times \R^d$.
In fact, the next proposition ensures that the map
$$\begin{array}{ccc}
 \Omega \times \R^{d} &\longmapsto &\R^{d}  \vsp \\
(t,\qqq, \tilde{\qqq}) & \longrightarrow &\P_{\Qc(t,\qqq)}(\tilde{\qqq})
\end{array}$$
 is continuous on the same neighbourhood (which is sufficient).
\label{localprop}
\end{rmrk}

\subsection{Proof of the claimed results}

\begin{proof}[Proof of Lemma~\ref{distQc}]
This proof is similar to the one developed for Proposition~\ref{Qlip}.
 Let $\uu $ given by Lemma \ref{gooddir} with $t=t_k^n $ and $\qqq= \qqq_k^n \in Q(t_k^n)$, we introduce 
 $\zz(s):=\qqq_k^n+s\uu$ with $s >0$. We claim that for $s \geq  \dsp\frac{2h\beta}{\delta} $ (where $c $ is introduced in (\ref{Ui})), 
$$
\dsp  \forall i\in I_\rho(t_k^n,\qqq_k^n), \qquad \Delta_i:=g_i(t_{k+1}^n,\qqq_k^n) + \ll \nabla_\qqq \, g_i( t_{k+1}^n,\qqq_k^n) \virg \zz(s)- \qqq_k^n\rr \geq 0.$$ 
Indeed, 
by Assumption (\ref{dtg}) and Proposition~\ref{schemewd}, $  g_i(t_{k+1}^n,\qqq_k^n) \geq g_i(t_{k}^n,\qqq_k^n) - \beta h \geq -\beta h$.
As in Proposition~\ref{schemewd}, it can be proved that for all $s \in [t_{k}^n,t_{k+1}^n ]$, $ \qqq_k^n \in U(s)$. Then it follows from Assumption (\ref{dtgradg}) that $$ \ll \nabla_\qqq \, g_i( t_{k+1}^n,\qqq_k^n) \virg \zz(s)- \qqq_k^n\rr \geq \ll \nabla_\qqq \, g_i( t_{k}^n,\qqq_k^n) \virg \zz(s)- \qqq_k^n\rr -M h s.$$
For $h \leq \frac{\delta}{2M} $ and by definition of $\uu$, $$\Delta_i \geq -\beta h -Mhs + s\delta \geq \dsp s \frac{\delta}{2}- h\beta\geq 0.$$
Moreover, for $  i\notin I_\rho(t_k^n,\qqq_k^n)$, Assumption (\ref{dtg}) yields $$ g_i( t_{k+1}^n,\qqq_k^n) \geq g_i( t_{k}^n,\qqq_k^n)-\beta h \geq \rho  -\beta h   .$$ Consequently for $  i\notin I_\rho(t_k^n,\qqq_k^n)$ and $h \leq \frac{\rho}{2\beta} $, $$ g_i( t_{k+1}^n,\qqq_k^n) +\ll \nabla_\qqq \, g_i( t_{k+1}^n,\qqq_k^n) \virg \zz(s)- \qqq_k^n\rr \geq \rho  -\beta h - \beta s \geq \frac{\rho}{2} - \beta s.
 $$
For $h$ small enough, for $s \in [\frac{2h\beta}{\delta}, \frac{\rho}{2\beta}]$, $\zz(s)$ belongs to $\Qc( t_{k+1}^n,\qqq_k^n) $.
 So we have proved that
 $$ d_{\Qc( t_{k+1}^n,\qqq_k^n)}(\qqq_k^n) \leq \inf_{s \delta  \geq 2 h\beta} |\qqq_k^n -\zz(s)| \leq \frac{2 \beta}{\delta}h.$$
\end{proof}

\begin{proof}[Proof of Lemma~\ref{convfaiblefn}]
We want to show that $\ff^n$ weakly converges to $ \ff(\cdot,\qqq(\cdot))$ in $L^1([0,T], \R^d)$.
This property rests on the boundedness of the sequence $ (\ff^n)_n$ in $L^\infty([0,T], \R^d)$.
Let us consider $\Psi \in L^\infty([0,T], \R^d) $ supposed to be continuous.
$$ \begin{array}{l}
\dsp \int_0^T \ll \ff^n(t), \Psi(t) \rr dt \vsp \\
 = \dsp \sum_{k=0}^{n-1} \int_{t_k^n}^{t_{k+1}^n} \ll \ff(t_k^n, \qqq(t_k^n) ), \Psi(t)\rr dt
+ \sum_{k=0}^{n-1} \int_{t_k^n}^{t_{k+1}^n} \ll \ff(t_k^n, \qqq^n(t_k^n) )-\ff(t_k^n, \qqq(t_k^n) ) ,\Psi(t) \rr dt 
 := I_n^1+ I_n^2. 
\end{array}  $$
By (\ref{lip}), $$|I_n^2| \leq KT \| \Psi\|_\infty \| \qqq^n-\qqq \|_\infty \xrightarrow[n \to +\infty]{} 0.$$ 
Furthermore
$$ \begin{array}{lll}
    I_n^1&=& \dsp \sum_{k=0}^{n-1} h \ll \ff(t_k^n, \qqq(t_k^n) ), \Psi(t_k^n) \rr 
 +\sum_{k=0}^{n-1} \int_{t_k^n}^{t_{k+1}^n} \ll \ff(t_k^n, \qqq(t_k^n) ), \Psi(t)-\Psi(t_k^n) \rr dt := J_n^1+ J_n^2. 
   \end{array}
$$
The map $\ff(\cdot, \qqq(\cdot))$ is bounded by $L(1+C)$ according to (\ref{lingro}) and (\ref{majq}). Let us check that 
\begin{equation} J_n^1\xrightarrow[n \to +\infty]{} \int_0^T \ll \ff(t, \qqq(t)), \Psi(t) \rangle dt .
 \label{limJ1}
\end{equation}
For all $\eta >0 $, there exists $m \in \mathbb{N} $ such that for all $t,s \in [0,T]$, $|\qqq(t)-\qqq(s)| < \eta $ if $|t-s| <\delta:=T/m$ since $\qqq \in W^{1, \infty}([0,T],Q) $. For all $t,s \in [0,T]$ satisfying $|t-s| <\delta$, (\ref{lip}) yields $| \ff(t,\qqq(t))-\ff(t,\qqq(s))| < K \eta $.
For $n >m $,
$$ J_n^1 = \sum_{k=0}^{n-1} h \ll \ff(t_k^n, \qqq(t_k^n) )-\ff(t_k^n, \qqq(\chi(t_k^n)) ), \Psi(t_k^n) \rr +  \sum_{k=0}^{n-1} h \ll\ff(t_k^n, \qqq(\chi(t_k^n)) ), \Psi(t_k^n) \rr, $$
where $ \dsp \chi(t) := \delta \left\lfloor \frac{t}{\delta} \right \rfloor$.
Obviously, $$ \left |\sum_{k=0}^{n-1} h \ll \ff(t_k^n, \qqq(t_k^n) )-\ff(t_k^n, \qqq(\chi(t_k^n)) ), \Psi(t_k^n) \rr   \right | \leq TK\eta \| \Psi\|_\infty.$$
Moreover $$ \sum_{k=0}^{n-1} h \ll \ff(t_k^n, \qqq(\chi(t_k^n)) ), \Psi(t_k^n) \rr = \sum_{i=0}^{m-1} \sum_{\genfrac{}{}{0pt}{}{k}{t_k^n \in [i\delta, (i+1)\delta[ }} h \ll \ff(t_k^n, \qqq(i \delta) ), \Psi(t_k^n) \rr. $$
Since for all $i \in\{0,...,m-1\} $, $\ff(\cdot, \qqq(i\delta))$ and $\Psi $ are bounded and Riemann integrable      by (\ref{riemann}), we have $$ \sum_{\genfrac{}{}{0pt}{}{k}{t_k^n \in [i\delta, (i+1)\delta[ }} h \ll \ff(t_k^n, \qqq(i \delta) ), \Psi(t_k^n) \rr \xrightarrow[n \to \infty]{} \int_{i \delta}^{(i+1)\delta} \ll \ff(t, \qqq(i \delta)), \Psi(t) \rr dt.  $$ 
Thus, as previously, $$ \limsup_{n \to \infty} \ \left|\sum_{k=0}^{n-1} h \ll \ff(t_k^n, \qqq(\chi(t_k^n)) ), \Psi(t_k^n) \rr -\int_{0}^{T} \ll \ff(t, \qqq(t)), \Psi(t) \rr dt \right| \leq TK \eta \| \Psi\|_\infty.  $$
This concludes the proof of (\ref{limJ1}). 
Moreover $$|J_n^2| \leq TL(1+C) \sup_{|t-s|\leq T/n} | \Psi(t)-\Psi(s) |\xrightarrow[n \to +\infty]{} 0 .$$ 
Then, we deduce that for all continuous map $\Psi \in L^\infty([0,T], \R^d) $, 
\begin{equation}
 \int_0^T \ll \ff^n(t), \Psi(t) \rr dt\xrightarrow[n \to +\infty]{}\int_0^T \ll \ff(t, \qqq(t)), \Psi(t) \rangle dt .
 \label{casreg}
\end{equation}

\noi Now let $\Phi \in L^\infty([0,T], \R^d) $, there exists a sequence of continuous maps $ (\Psi_p)_p$ satisfying $$ \| \Psi_p -\Phi\|_1 \xrightarrow[p \to +\infty]{} 0 .$$
Let $\varepsilon >0 $, there exists $p \in \mathbb{N} $ such that  $ L(1+C)\| \Psi_p -\Phi\|_1 \leq \varepsilon  $. For that integer $p$, we write
$$
\begin{array}{l}
 \dsp \left| \int_0^T \ll \ff^n(t), \Phi(t) \rr dt - \int_0^T \ll \ff(t, \qqq(t)), \Phi(t) \rr dt  \right| \vsp \\ 
\dsp \leq \left| \int_0^T \ll \ff^n(t)-\ff(t, \qqq(t)) ,\Psi_p(t) \rr dt +  \int_0^T \ll \ff(t, \qqq(t)),\Psi_p(t)-\Phi(t)  \rr dt + \int_0^T \ll \ff^n(t), \Phi(t)-\Psi_p(t) \rr dt  \right| \vsp \\ 

\dsp \leq  \int_0^T \left| \ll \ff^n(t)-\ff(t, \qqq(t)) ,\Psi_p(t) \rr \right|dt +  \int_0^T \left|\ll \ff(t, \qqq(t)),\Psi_p(t)-\Phi(t)  \rr\right| dt + \int_0^T\left| \ll \ff^n(t), \Phi(t)-\Psi_p(t) \rr \right|dt 

\vsp \\ \leq K_1 + K_2 +K_3.
\end{array}
$$
For $i=2$ and $3$, $ K_i \leq L(1+C)\| \Psi_p -\Phi\|_1 \leq \varepsilon  $ and by (\ref{casreg}), we know that for $n $ large enough, $K_1 \leq \varepsilon$. So we deduce that for all $\Phi \in L^\infty([0,T], \R^d) $, 
\begin{equation*}
\int_0^T \ll \ff^n(t), \Phi(t) \rr dt\xrightarrow[n \to +\infty]{}\int_0^T \ll \ff(t, \qqq(t)), \Phi(t) \rangle dt ,
\end{equation*} 
which concludes the proof.
\end{proof}

\begin{prpstn}
Let $t\in[0,T] $, $\qqq \in Q(t) $, a
sequence $(t_n) \in [0,T]^\mathbb{N}$ converging to $t $ and a
sequence $(\qqq_n ) \in (U(t_n))_n $ converging to $\qqq$ such that $d_{\Qc(t_n,\qqq_n)}(\qqq_n) $ tends to zero. For all  $\tilde{\qqq} \in
\R^{d} $, we denote $\tilde{\pp}$ (respectively
  $\tilde{\pp_n}$) the projection of $\tilde{\qqq}$ onto $\Qc(t,\qqq) $
(respectively onto $\Qc(t_n,\qqq_n)$). Then there exists $ \nu >0 $ so that
for all $\tilde{\qqq} \in
B(\qqq, \nu) $, the sequence $(\tilde{\pp_n}) $
converges to  $\tilde{\pp}$.
\label{lem:proj}
\end{prpstn}
\begin{proof}
By Lemma~\ref{point-selle}, the vector $ \tilde{\pp}$ satisfies System~(\ref{pqqt}).
We obtain similar systems denoted by $\left(P_{t_n,q_n,\tilde{q}}\right)$ for all $\tilde{\pp_n}$ in substituting $t$, $\qqq $
 and $\lambda_{i} $ by $t_n $, $\qqq_n $ and $\lambda_{i}^n $. 
The following lemma (which will be later proved) claims that the nonzero Kuhn-Tucker multipliers $\lambda_i, \  \lambda_i^n $ are associated to an index $i $ belonging to $I(t,\qqq)$ (defined by (\ref{def:I})).

\begin{lmm}
There exist $\nu > 0 $ and $M_0 \in \mathbb{N} $ such that for all
$ n \geq M_0$ and all $\tilde{\qqq} \in B(\qqq, \nu)$, we have
:
$$\dsp |\tilde{\pp_n}-\tilde{\qqq} |\leq 2\nu, $$ and 
$$ \dsp \lambda_{i}^n 
=\lambda_{i}= 0, \textmd{ if } g_i(t,\qqq) >0 ,$$ where $(\tilde{\pp}, \bflambda)$ and $(\tilde{\pp_n}, \bflambda^n)$ are
respectively solutions of
$\left(P_{t,q,\tilde{q}}\right)$ and 
$\left(P_{t_n,q_n,\tilde{q}}\right)$.
\label{nul}
\end{lmm}
\noi There exists $M_1 \in \mathbb{N}$ such that for $n\geq M_1$, $|\qqq-\qqq_n|<c .$ 
 For all $s \in [0,T]  $ and $\qqq,\ \xx \in Q(t) $, we denote by $A_{t,\qqq}(s,\xx)$ the $d \times |I(t,\qqq)|  $ matrix defined as follows:
$$A_{t,\qqq}(s,\xx):=\left( \nabla_\qqq \, g_i(s,\xx) \right)_{i \in I(t,\qqq)}.$$ 
Let $M_0 $ and $\nu $ be fixed by Lemma~\ref{nul}.
With the previous notation, for $n\geq \max(M_0, M_1)$ and $ \tilde{\qqq } \in B(\qqq, \nu)$, the first equation of $\left(P_{t_n, q_n,\tilde{q}}\right)$ can be written as $$A_{t,\qqq}(t_n,\qqq_n)[\bflambda^n] =\tilde{\pp_n}-\tilde{\qqq}. $$
By (\ref{gradg}) and (\ref{inegtrianginverse}), we have $$ |\bflambda^n| \leq \dsp \frac{\gamma}{\alpha} \left|A_{t,\qqq}(t,\qqq)  [\bflambda^n]\right|, $$
where $|\bfmu| $ represents the Euclidean norm of $ \bfmu \in \R^{|I(t,\qqq)|}$.
Moreover $$ |A_{t,\qqq}(t,\qqq)  [\bflambda^n]| \leq |\tilde{\pp_n}-\tilde{\qqq} | + |(A_{t,\qqq}(t_n,\qqq_n) -A_{t,\qqq}(t,\qqq))[\bflambda^n] |.$$
Thus by Lemma~\ref{nul}, $$ |\bflambda^n| \leq \dsp \frac{\gamma}{\alpha} \left(2 \nu +  \varepsilon_n |\bflambda^n| \right),$$ with $ \varepsilon_n: = \|(A_{t,\qqq}(t_n,\qqq_n) -A_{t,\qqq}(t,\qqq) \|.$
Hence $$ |\bflambda^n| \left(1-\frac{ \varepsilon_n \gamma }{\alpha}\right) \leq  \frac{2 \nu \gamma }{\alpha }.$$
Furthermore, $\varepsilon_n \leq \sqrt{p} \max_{j \in I(t,\qqq)} |\nabla_\qqq \, g_j(t_n,\qqq_n)- \nabla_\qqq \, g_j(t,\qqq) |. $
Since $\qqq_n \in U(t) $, it can be written
$$ \begin{array}{lll}
|\nabla_\qqq \, g_j(t_n,\qqq_n)- \nabla_\qqq \, g_j(t,\qqq) |& \leq &|\nabla_\qqq \, g_j(t_n,\qqq_n)- \nabla_\qqq \, g_j(t,\qqq_n) |+ 
|\nabla_\qqq \, g_j(t,\qqq_n)- \nabla_\qqq \, g_j(t,\qqq) |\vsp \\ &\leq &M(|t-t_n|+ |\qqq- \qqq_n|),
\end{array}
$$
by Assumptions~(\ref{dtgradg}) and (\ref{hessg}).
Finally  $$\varepsilon_n \leq \sqrt{p}M(|t-t_n|+ |\qqq- \qqq_n| ) \xrightarrow[n \to \infty]{}0$$ and so we deduce that the sequence $(\bflambda^n)$ is bounded. A convergent subsequence can be also extracted and by passing to the limit in System $\left(P_{t_n, q_n,\tilde{q}}\right)$, we obtain that the corresponding subsequence of $(\tilde{\pp_n})$ 
converges to a point $\pp_\infty$ which satisfies System $\left(P_{t, q,\tilde{q}}\right)$. Lemma~\ref{point-selle} implies that $\pp_\infty =\tilde{\pp} $ and we can conclude by compacity arguments. 
\end{proof}
\noi It remains to prove Lemma~\ref{nul}.
\begin{proof}[Proof of Lemma~\ref{nul}]
This result is a consequence of the third equation of problems $\left(P_{t,q,\tilde{q}}\right)$ and
$\left(P_{t_n, q_n,\tilde{q}}\right)$.
By definition of  $ I(t,\qqq)$,$$  \dsp \exists
\varepsilon >0, \  \forall \, i \notin I(t,\qqq), \ \ g_i(t,\qqq) > 2
\varepsilon. $$
 Setting $ \dsp \nu =\frac{\varepsilon}{8 \beta}$, as $(\qqq_n)_n$ and $(t_n)_n $ respectively converge to $\qqq$ and $t$, it yields
$$ \dsp \exists
M_0 >0, \forall n \geq M_0, \ g_i(t_n,\qqq_n) \geq \varepsilon, \ \forall \,i \notin I(t,\qqq) \virg d_{\Qc(t_n,\qqq_n)}(\qqq_n) \leq \frac{\nu}{2} \textmd{ and }
|\qqq_n -\qqq | \leq \frac{\nu}{2}.$$
 Let
 $\tilde{\qqq} \in B(\qqq, \nu) $ , since $ \dsp \tilde{\pp} =  \textrm{P}_{\Qc(t,\qqq)} (\tilde{\qqq}) $ and
 $ \dsp \qqq \in \Qc(t,\qqq) $ (because $\qqq \in Q(t) $),
we obtain
$$
 \dsp |\tilde{\pp}-\tilde{\qqq} |\leq |\qqq -\tilde{\qqq}| \leq \nu $$
and consequently,
\begin{equation}
|\tilde{\pp}- \qqq|\leq  |\tilde{\pp}-\tilde{\qqq} |
+|\tilde{\qqq}-\qqq| \leq 2|\tilde{\qqq}-\qqq| \leq 2 \nu.
\label{inegproj}
\end{equation}
Moreover, as $ \dsp \tilde{\pp_n} = \textrm{P}_{\Qc(t_n,\qqq_n)}
(\tilde{\qqq}) $,
we get
$$
\dsp |\tilde{\pp_n}-\tilde{\qqq} |\leq |\qqq_n
-\tilde{\qqq}| + d_{\Qc(t_n,\qqq_n)}(\qqq_n)\leq |\tilde{\qqq}- \qqq |+|\qqq-\qqq_n|+d_{\Qc(t_n,\qqq_n)}(\qqq_n) \leq \nu+\frac{\nu}{2}+\frac{\nu}{2}=2\nu.
$$
 Hence,
\begin{equation}
|\tilde{\pp_n}- \qqq_n|\leq  |\tilde{\pp_n}-\tilde{\qqq} |
+|\tilde{\qqq}- \qqq_n| \leq 2|\tilde{\qqq}- \qqq_n| \leq 
 4 \nu 
.
\label{inegprojn}
\end{equation}
For $i \notin I(t,\qqq) $, it follows from (\ref{gradg}), (\ref{inegproj}) and (\ref{inegprojn})  $$\dsp g_i(t,\qqq)+ \langle \nabla_\qqq \, g_i(t,\qqq),
\tilde{\pp}-\qqq \rangle \geq 2 \varepsilon - 2\nu \beta >0
\textmd{ and }
 \dsp g_i(t_n,\qqq_n)+ \langle \nabla_\qqq \, g_i(t_n,\qqq_n),
\tilde{\pp_n}-\qqq_n \rangle \geq \varepsilon - 4 \nu \beta >0.$$
The third equation of problems $\left(P_{t,q,\tilde{q}}\right), \ \left(P_{t_n,q_n,\tilde{q}}\right)$ and the nonnegativity of Kuhn-Tucker multipliers $\lambda_i, \  \lambda_i^n $ permit us to conclude.
\end{proof}


\section{Applications}
\label{sec:appli}
\subsection{A case in point}
Here we deal with a simple case associated to a well-known game: the labyrinth tabletop game. The aim of this game is to maneuver a steel ball through a wooden labyrinth by tilting the surface in order to move the ball to the target. The angle of the labyrinth need to be carefully controlled so that the ball will not fall into the holes. \\
We consider a particular labyrinth represented in Figure~\ref{fig:laby}, with several holes (corresponding to the black disks).
\begin{figure}
\centering
\begin{tabular}{c}
\psfrag{A}[l]{S}
\psfrag{B}[l]{T}
\psfrag{C}[l]{P$_1$}
\psfrag{D}[l]{P$_2$}
\psfrag{f1}[l]{$\ff_1$}
\psfrag{f2}[l]{$\ff_2$}
\psfrag{f3}[l]{$\ff_3$}
\includegraphics[width=0.65\textwidth]{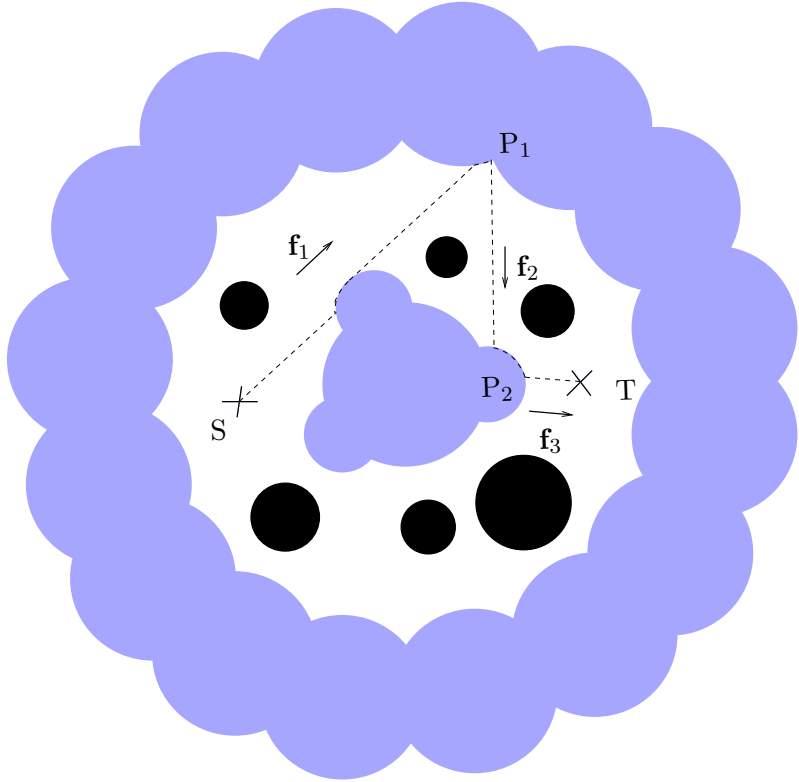}
\end{tabular}
\caption{Labyrinth game.}
\label{fig:laby}
\end{figure}
Initially the ball is located at the starting point S and it has to be rolled to the target T.
The quasi-static evolution of this ball can be described by the following first order differential inclusion: $$ \dsp \frac{d\qqq}{dt}(t) + \NN(Q,\qqq(t) ) \ni \ff(t),$$ where $\qqq \in \R^2$ is the position of the ball in the labyrinth surface, $Q$ is the complementary set of the grey obstacles and the holes and $\ff(t) $ is the perturbation caused by the leaning of the labyrinth plane. In Figure~\ref{fig:laby}, we plotted the trajectory obtained with the following field $\ff$:
beginning with $\ff=\ff_1$, the ball arrives at the point P$_1$, then with $\ff=\ff_2$, the ball goes to the point P$_2$ and finally submitted to $\ff=\ff_3 $ it reaches the target. \\

\noi Obviously, the set $Q$ is uniformly prox-regular and Assumptions~(\ref{gradg}), (\ref{hessg}) and (\ref{inegtrianginverse}) are satisfied. The evolution problem is well-posed by Theorem~\ref{theo:wp} and the trajectory of the ball can be estimated by the numerical scheme presented in Section~\ref{sec:schema}.

\subsection{An application to a crowd motion model}
\label{ssec:crowd}
The aim of this subsection is to apply the previous results to a model of crowd motion in emergency evacuation. We refer the reader to \cite{esaim,CRAS,thesejv,M2AN} for a complete and detailed description of this model. 

\subsubsection{Presentation of the model}

\mb We quickly recall the model. Its asset is to directly handle numerous contacts which are characteristic of emergency situations. This model allows us to deal with local interactions between people in order to describe the whole dynamics of the pedestrian traffic. This microscopic model for crowd motion rests on two principles. On the one hand, each individual has a spontaneous velocity that he would like to have in the absence of other people. On the other hand, the actual velocity must take into account congestion. Those two principles lead to define the actual velocity field as the projection of the spontaneous velocity over the set of admissible velocities (regarding the non-overlapping constraints).

\mb We consider $N$ persons identified to rigid disks. For convenience, the disks are supposed here to have the same radius $r$. The
center of the $i$-th disk is denoted by $\qq_i\in \R^2$. Since overlapping is forbidden, the  
vector of positions $\qqq=(\qq_1,..,\qq_N) \in
\R^{2N}$ has to belong to the ``set of feasible configurations'', defined by 
\begin{equation}
\label{eq:Q} Q:=\left\{ \qqq  \in
\R^{2N},\ D_{ij}( \qqq) \geq 0\quad   \forall \,i \neq j  \right\},  
\end{equation}
where $ D_{ij}(\qqq)=|\qq_i-\qq_j|-2r $ is the signed distance between disks $i$ and $j$. 

\mb We denote by $\UU(\qqq)=(U_1(\qq_1),..,U_N(\qq_N)) \in \R^{2N}$ the global spontaneous velocity of the crowd.
We introduce the ``set of feasible velocities'' defined by:
 $$ \CCC_{\qqq}=\left\{  \vv  \in
\R^{2N} , \ \forall i<j \hspace{5mm} D_{ij}(\qqq)=0 \hspace{3mm}
\Rightarrow \hspace{3mm}
\langle \GG_{ij}(\qqq),\vv \rangle \geq 0 \right\},  $$ with
$$\GG_{ij}(\qqq)=\nabla D_{ij}(\qqq)=
(0,\dots,0, -\ee_{ij}(\qqq),0,\dots,0,\ee_{ij}(\qqq),0,\dots,0 )\in
\R^{2N}  $$ and $\ee_{ij}(\qqq)= \frac{\qq_j-\qq_i}{|\qq_j-\qq_i|}$.
The actual velocity field is defined as the feasible field which is the closest to $\UU$ in the least square sense, which writes
\begin{equation}
 \label{eq:model1}
\frac{d\qqq}{dt} = \P_{\CCC_{\qqq}}\left(\UU(\qqq)\right),
\end{equation}
where $\P_{\CCC_{\qqq}}$ denotes the Euclidean projection onto the closed convex cone $\CCC_{\qqq}$.
We deduce from Farkas Lemma (see e.g.~\cite{Ciarlet}) and Proposition~\ref{coneprox} the following results:
\begin{prpstn} \label{prop:model1}  The negative polar cone $\NNN_\qqq$ of $\CCC_\qqq $ defined by
$$ \NNN_\qqq := \CCC_\qqq^\circ
:=  \left\{\ww \in \R^{2N},\ \langle \ww , \vv \rangle \leq 0 \quad \forall \vv \in \CCC_\qqq \right\},$$
is equal to the proximal normal cone $\NN(Q,\qqq)$. More precisely,
$$ \NNN_\qqq = \NN(Q,\qqq) = \left\{-\sum \lambda_{ij} \GG_{ij}(\qqq),\ \lambda_{ij} \geq 0,\ 
D_{ij}(\qqq) > 0 \Longrightarrow \lambda_{ij} = 0
\right\} .$$
\end{prpstn}

\mb Using the classical orthogonal decomposition with two mutually polar cones (see \cite{Moreaucones}), the main equation (\ref{eq:model1}) becomes
\begin{equation}
\frac{d\qqq}{dt} + \P_{\NN(Q,\qqq)}\left(\UU(\qqq)\right) = \UU(\qqq).
 \label{eq:model2}
\end{equation}
According to Proposition 3.3 in~\cite{Fred}, we know that for a Lipschitz map $\UU$, this differential equation is equivalent to the following differential inclusion:
\begin{equation}
\label{eq:model3} 
\frac{d\qqq}{dt} + \NN(Q,\qqq) \ni \UU(\qqq).
\end{equation}

\noi For all $(i,j)$, $D_{ij} $ is a convex function and belongs to $C^2(U_{ij})$, with  $$ U_{ij}:=\{ \qqq \in \R^{2N} \virg |\qq_i-\qq_j |-r>0\} $$
and satisfies Assumption (\ref{Ui}) with $c=r/\sqrt{2}$. Moreover it is obvious that $$ \forall \qqq \in U_{ij} \virg |\GG_{ij}(\qqq)|=\sqrt{2} \quad \textmd{     and     } \quad |\textrm{D}^2D_{ij}(\qqq)| \leq \frac{2}{r}   .$$
As a consequence, Assumptions~(\ref{gradg}) and (\ref{hessg}) are satisfied.
It remains to check the inverse triangle inequality (Assumption~(\ref{inegtrianginverse})), which is the aim of the following proposition.

\begin{prpstn}[Inverse triangle inequality] \mbox{}
\newline There exists $\gamma >1  $ such that for all $\qqq \in Q$,
$$\sum_{(i,j) \in I(\qqq)} \alpha_{ij}|\GG_{ij}(\qqq)| 
\leq \gamma \left|\sum_{(i,j) \in I(\qqq)} \alpha_{ij}\GG_{ij}(\qqq)
\right|, $$ where $$ I(\qqq)=\{(i,j), \ i<j,\  D_{ij}(\qqq)=0 \} 
\textmd{ and } \alpha_{ij} \textmd{ are nonnegative reals}. $$
Constant $\gamma $ can be fixed as follows $$\gamma = \dsp 3 \sqrt{2} N \left(\frac{3}{\sin \left( \frac{2\pi}{N}\right)} \right)^N.$$
\label{propinegtrianginv}
\end{prpstn}

\noi The next subsection is devoted to its proof.

\begin{rmrk}
 Note the sign of coefficients $\alpha_{ij}$. From a general point of view, this inequality is obviously wrong if  these coefficients are just assumed real. Indeed, for $N \geq 6 $, the cardinal of the set $I(\qqq)$ is strictly larger than  $2N$, which involves a relation between vectors $\GG_{ij}(\qqq)$.  
\end{rmrk}

\begin{rmrk}
In~\cite{M2AN}, we have already proved such a result. The proof we propose here gives a smaller constant $\gamma$ and above all, it allows to better understand how the non-uniqueness of the Kuhn-Tucker multipliers appears.    
\end{rmrk}

\noi Proposition~\ref{Qprox}, Theorems~\ref{theo:wp2} and \ref{theo:qhq} imply the next results.

 \begin{prpstn} The set $Q\subset \R^{2N}$, defined by (\ref{eq:Q}) is $\eta$-prox-regular with a constant $$\eta= \frac{r}{6N} \left(\frac{\sin\left(\frac{2\pi}{N}\right)}{3} \right)^N.$$
  \end{prpstn}

\begin{thrm}
 Assume that $\UU$ is Lipschitz. The Cauchy problem associated to (\ref{eq:model2}) is well-posed in the set of the absolutely continuous functions and the related numerical scheme (\ref{schema}) is convergent.
\end{thrm}

\noi Other numerical simulations based on this scheme, we refer the reader to~\cite{esaim, MauryVenelTGF, justrat, M2AN}.
%


\subsubsection{Proof of the inverse triangle inequality}
In order to prove Proposition~\ref{propinegtrianginv}, we are firstly going to show that the terms $ \lambda_{ij} $ appearing in the following equation : $$ \sum_{i<j} \lambda_{ij}
\GG_{ij}(\qqq)= \FF$$ are bounded.

\begin{prpstn}
For all $ \qqq \in Q$, for all $\FF \in \R^{2N} $, 
 the following set
$$\Lambda_{\qqq,F}:=\left\{\bflambda \in  \R^{\frac{N(N-1)}{2}}, \sum_{i<j} \lambda_{ij}
  \GG_{ij}(\qqq) =\FF ,\ \lambda_{ij}\geq 0, \lambda_{ij}=0 \textmd{ if }
 D_{ij}(\qqq)>0 \right\} $$ is uniformly bounded with respect to $\qqq$. More precisely, $$  \dsp \forall \bflambda \in \Lambda_{\qqq, F}, \ 
\forall i<j, \  \lambda_{ij} \leq |\FF|\ a^N  
\textmd{ with } a = \frac{3}{\sin(\frac{2\pi}{N})}. $$
\label{lambdaq}
\end{prpstn}
\noindent First we check that Proposition~\ref{lambdaq} implies Proposition~\ref{propinegtrianginv}.
Indeed, by Proposition~\ref{lambdaq}, we have for all $(k,l) \in I(\qqq)$,
$$  \alpha_{kl} \leq a^N  \ \left| \sum_{(i,j) \in I(\qqq)} \alpha_{ij}\GG_{ij}(\qqq) \right|, $$
with $$ a= \dsp  \frac{3}{\sin \left( \frac{2\pi}{N}\right)}  .$$
Since $|\GG_{kl}(\qqq)|=\sqrt{2} $, by summing coefficients $\alpha_{kl}  $ for all $(k,l) \in I(\qqq)$ , we obtain
$$ \sum_{(k,l) \in I(\qqq)} \alpha_{kl} |\GG_{kl}(\qqq)| \leq \sqrt{2} \  |I(\qqq)| \  a^N \  \left| \sum_{(i,j) \in I(\qqq)} \alpha_{ij}\GG_{ij}(\qqq) \right| .$$ 
In the monodisperse case, each disk has at most 6 neighbours. As a consequence, $ |I(\qqq)| \leq 3N ,$ and thus, $$ \sum_{(k,l) \in I(\qqq)} \alpha_{kl} |\GG_{kl}(\qqq)| \leq 3\sqrt{2}N \  a^N \  \left| \sum_{(i,j) \in I(\qqq)} \alpha_{ij}\GG_{ij}(\qqq) \right| . $$ 
It now remains to prove Proposition~\ref{lambdaq}.

\begin{proof}[Proof of Proposition~\ref{lambdaq}]
Suppose that set $\Lambda_{\qqq,F}$ is not empty, we want to estimate the solutions $\bflambda$ of the following system containing $2N$ equations:  
  \begin{equation}
\tag{$P$}
\sum_{\footnotesize{
\begin{array}{c}
i<j \\
D_{ij}(\qqq)=0 \\
\end{array}}
} \lambda_{ij} \GG_{ij}(\qqq) = \FF \virg \lambda_{ij}\geq 0. 
 \end{equation}
 By specifying the expression of vectors $\GG_{ij}(\qqq) $, we can write the system concerning individual $i_0$, 
\begin{equation}
\tag{$P_{i_0}$} 
\sum_{\footnotesize{
\begin{array}{c}
j=1 \\
j \neq i_0 \\
j \textmd{ neighbour of }i_0
\end{array}}
}^n \overline{\lambda_{ji_0}} e_{ji_0}= F_{i_0},
 \end{equation}
 where
 $$ \dsp \overline{\lambda_{ji_0}} = \left \{
\begin{array}{l}
\lambda_{ji_0} \textmd{ if } j<i_0 \\
\lambda_{i_0j} \textmd{ if } j>i_0 
\end{array}
\right. \textmd{ and } e_{ji_0}= \ee_{ji_0}(\qqq) .$$
We assume here that the configuration $\qqq \in Q$ is a cluster in the sense that the set $\bigcup_{i=1}^{N} \overline{B(\qq_i,r)}$ is connected by arcs. Otherwise, for the other configurations, it suffices to deal with the different clusters one by one, as the associated problems are independent from each other.  
 We denote by $A=\{\qq_1,.., \qq_N \}$ the set of all positions. We are going to solve the problems ($P_{i}$) one by one. After having solved a problem ($P_{i_0}$), $\qq_{i_0}$ is removed from $A$ and problem ($P_{i_0}$) is deleted from the whole problem $(P)$, and the values of  $\overline{\lambda_{ji_0}}$ are taken into account in terms $F_j $, for each person $j$ neighbour of individual $i_0 $.
 The goal is to reduce $A$ to a singleton. \\ 
The order of this algorithm is important. Less terms there are in problem $(P_{i_0})$, more easily solvable it is. First and foremost, we deal with people having only one neighbour (Case 1) because in that case the solution is trivial. When there is no single contact, we consider people with more neighbours. Since we want to have an upper bound for the terms $ \overline{\lambda_{ji_0}}$, we will take care of controlling the angles between vectors $e_{i_0j_1} $ and $e_{i_0j_2} $ if individuals $ j_1$ and $j_2$ are in contact with person $i_0$ (see Case 2). 
\bigskip \\
{\bf Case 1: elimination of single contacts}\\
 Suppose that there exists $\qq_i \in  A$ such that person $i $ has only one neighbour $j$.
 Problem $(P_i)$ becomes 
$$\overline{\lambda_{ji}} e_{ji}= F_i. $$ Consequently, $\overline{\lambda_{ji}}= |F_i|$. We remove $\qq_i $ from $A$ and $(P_i)$ from $(P)$. Then, we substitute $F_j $ by $F_j-\overline{\lambda_{ji}} e_{ij} $, which implies that the term $|F_j|$ is replaced with $|F_j|+ |F_i| $. 
In this way, all single contacts are taken off, which allows us for example to completely deal with the case illustrated in Figure~\ref{fig:chaine}. 
\begin{figure}
\centering
\begin{tabular}{c}
\includegraphics[width=0.25\textwidth]{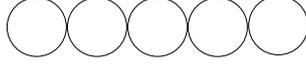}
\end{tabular}
\caption{Case of a chain.}
\label{fig:chaine}
\end{figure}
If $A $ is not reduced to a singleton after these eliminations, we consider the following case.
\bigskip \\
{\bf Case 2: there are no single contacts}\\
We define $C$ the convex hull of $ A$. The set of its extremal points is denoted by $ E$ and the cardinal of $ E$ by $p$. We denote by $P$ the boundary of $C$, which is a convex polygon whose vertices are the points of $E$. This polygon has $p \geq 3$ corners at least, otherwise we would rather consider Case 1. The angle sum equals to $\pi (p-2) $. Therefore there exists an angle of $P$ lower than $\pi (1-\frac{2}{p})< \pi$. Suppose that $\qq_i $ is the corresponding vertex. Because of the angle condition, individual $i$ has two or three neighbours.
Indeed, if individuals $j $ and $k$ are in contact with person $i$, the minimum angle between vectors $e_{ij}$
and $e_{ik}$ is equal to $\frac{\pi}{3} $.     
\bigskip \\
{\bf Case 2 (a)}: individual $i $ has 2 neighbours \\
Using a rotation around $\qq_i $ if needed, we can assume that the concerned configuration is represented in Figure~\ref{fig:cas1}. In view of the choice of $\qq_i$, $ \theta \geq
\frac{2\pi}{p} \geq \frac{2\pi}{N} $ and $ \theta \leq  \frac{2\pi}{3}$. Note that 
$N\geq p\geq 3$.  
\begin{figure}
\centering
\begin{tabular}{c}
\psfrag{a}[l]{$\qq_i$}
\psfrag{b}[l]{$\qq_j $}
\psfrag{c}[l]{$\qq_k$}
\psfrag{d}[l]{$\theta$}
\includegraphics[width=0.25\textwidth]{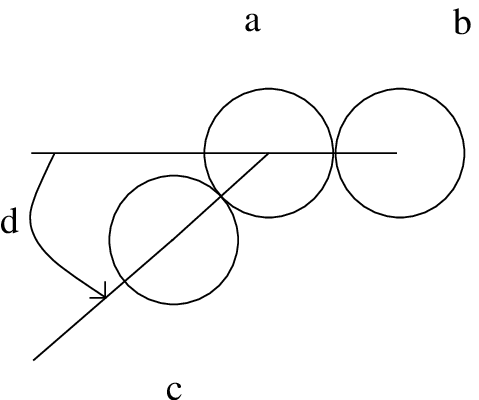}
\end{tabular}
\caption{Case 1.}
\label{fig:cas1}
\end{figure}
By writing the expression of $e_{ji}$ and $e_{ki}$, we can solve Problem $(P_i)$ after a simple computation:
$$
\left (
\begin{array}{c}
\overline{\lambda_{ji}} \vsp\\ \overline{\lambda_{ki}} 
\end{array}  \right )= \frac{1}{\sin \theta} 
\left (
\begin{array}{cc}
-\sin \theta &  \cos \theta \vsp\\
0 & 1
\end{array}  \right ) 
\left (
\begin{array}{l}
F_i^x \vsp \\ F_i^y 
\end{array}  \right ).
$$
Hence $$\overline{\lambda_{ji}} \leq \frac{\sqrt{2}}{\sin \theta}
|F_i| \textmd{ and }\overline{\lambda_{ki}} \leq \frac{1}{\sin \theta}
|F_i|.$$
 We remove $\qq_i $ from $A $ and we replace $|F_ j|$ (respectively $|F_ k| $) with 
 $ \dsp |F_j|+ \sqrt{2} |F_i|/\sin (2\pi/N)$ (respectively $ \dsp |F_k|+ 
|F_i|/ \sin(2\pi/N)$) because $\sin \theta \geq \sin(\frac{2\pi}{N}) $.
Problem $(P_i) $ is also eliminated.
\bigskip \\
{\bf Case 2 (b)}: individual $i $ has 3 neighbours \\
Using a rotation around $\qq_i $ if needed, we can assume that the concerned configuration is represented in Figure~\ref{fig:cas2}.
Because of the non-overlapping constraint, angles $\theta  $ and $\phi$ satisfy the following double inequalities 
\begin{equation}
 \frac{2\pi}{N} \leq \theta \leq \frac{\pi}{3}.
\label{eq:gamma}
\end{equation}
and
\begin{equation}
 \frac{2\pi}{N}- \frac{2\pi}{3} \leq \phi \leq -\frac{\pi}{3}.
\label{eq:beta}
\end{equation}
Note that $N \geq p \geq 6$. In fact, for $N<7$ persons, we can see that the cluster $A$ can be reduced to a singleton only by considering Case 1.   
\begin{center}
\begin{figure}
\centering
\begin{tabular}{c}
\psfrag{a}[l]{$\qq_i$}
\psfrag{b}[l]{$\qq_j $}
\psfrag{c}[l]{$\qq_l$}
\psfrag{e}[l]{$\qq_k$}
\psfrag{d}[l]{$\theta$}
\psfrag{f}[l]{$\phi$}
\includegraphics[width=0.25\textwidth]{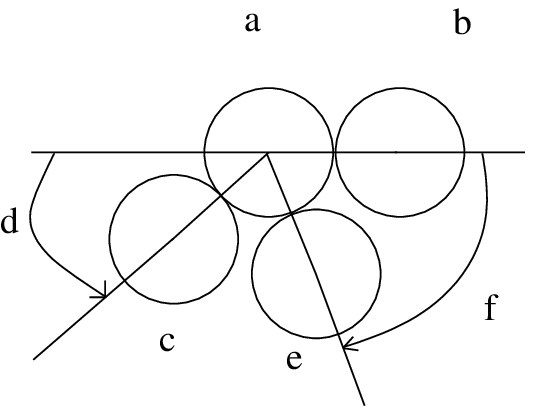}
\end{tabular}
\caption{Case 2.}
\label{fig:cas2}
\end{figure}
\end{center}
In this case, the terms $\overline{\lambda_{ji}}$, $\overline{\lambda_{ki}}$ and $\overline{\lambda_{li}}$ are not unique (3 unknowns and 2 equations). We define a particular solution of Problem ($P_i$)
 $(\overline{\lambda_{ji}^0},0,\overline{\lambda_{li}^0} )$, where
\begin{equation} 
\left (
\begin{array}{c}
\overline{\lambda_{ji}^0} \vsp \\ \overline{\lambda_{li}^0} 
\end{array}  \right )= \frac{1}{\sin \theta} 
\left (
\begin{array}{cc}
-\sin \theta &  \cos \theta \vsp \\
0 & 1
\end{array}  \right ) 
\left (
\begin{array}{l}
F_i^x \vsp \\ F_i^y 
\end{array}  \right ).
\label{solpart}
\end{equation}
To obtain all solutions, it remains to describe the kernel of System 
$(P_i) $. We can easily show the next lemma. 
\begin{lmm}
The kernel of System $(P_i) $ is generated by: $$k_{\phi \theta}=\left (
\begin{array}{c}
\sin(\phi-\theta) \\ \sin\theta  \\ \sin\phi
\end{array}  \right ). $$ In addition, the signs of the coordinates can be specified   $$ \sin(\phi-\theta) \leq -\sin\left(\frac{2\pi}{N}\right)< 0,\ 
 \sin\theta \geq \sin\left(\frac{2\pi}{N}\right) >  0 \textmd{ and }  \sin\phi \leq
-\frac{\sqrt{3}}{2} < 0. $$ 
\label{resolsyst} 
\end{lmm}
\noindent All the solutions of $(P_i) $ take the form $$
\left (
\begin{array}{c}
\overline{\lambda_{ji}} \vsp \\\overline{\lambda_{ki}} \vsp \\ \overline{\lambda_{li}}
\end{array}  \right )=
\left (
\begin{array}{c}
\overline{\lambda_{ji}^0} \vsp \\0  \vsp \\ \overline{\lambda_{li}^0}
\end{array}  \right )+t
\left (
\begin{array}{c}
\sin(\phi-\theta) \vsp \\ \sin\theta  \vsp\\ \sin\phi
\end{array}  \right )  \textmd{ where } t \in \R.
 $$
Since we are only looking for the nonnegative solutions, the sign of $\sin
\theta $ implies $t \geq 0 $. Furthermore, the signs of
$\sin(\phi-\theta)  $ and of $\sin\phi $ involve  $t\leq
t_\textmd{max} $  where
\begin{equation} 
\dsp t_\textmd{max} = 
\min\left (\frac{\overline{\lambda_{ji}^0}}{-\sin(\phi-\theta)}, 
     \frac{\overline{\lambda_{li}^0}}{-\sin\phi} \right ).
\label{tmax} 
\end{equation} 
Moreover, the following inequalities are satisfied:
$$\dsp 
\overline{\lambda_{ji}} \leq \frac{\sqrt{2}}{\sin \theta}
|F_i|, \quad
\overline{\lambda_{li}} \leq \frac{1}{\sin \theta}
|F_i| \quad \textmd{ and } \quad \overline{\lambda_{ki}} \leq t_\textmd{max}\sin\theta \leq \frac{2}{\sqrt{3}}
|F_i| .$$
So we remove $\qq_i$ from $A $ and we replace $|F_ j|$ (respectively $|F_ k| $ and $|F_l| $) with 
 $ \dsp |F_j|+\sqrt{2}|F_i|/\sin (2\pi/N)$ (respectively $ \dsp|F_k|+ 
2|F_i|/\sqrt{3}$ and $\dsp |F_l|+ |F_i|/\sin(2\pi/N) $) .
Problem $(P_i) $ is also eliminated.\\
Now we have removed one person from set $A$. If this set is not reduced to a singleton, we return to Case 1 or Case 2.
\bigskip \\
{\bf Conclusion}\\
By iterating this process, we solve Problem $(P)$ with at most $(N-1) $ steps.
We can conclude with the following lemma, whose the proof is a straightforward computation.
\begin{lmm}
At every step, the norm $|\FF|$ is replaced at worst with $ \dsp \frac{3}{\sin(\frac{2\pi}{N})} |\FF| $. 
\label{bilan}
\end{lmm} 
\noi This lemma ends the proof of Proposition~\ref{lambdaq}.
\end{proof}
\begin{rmrk}
A similar result can be proved in the polydisperse situation (where the radii are not assumed equal). Case 2 presents also more possibilities because the maximum number of neighbours that a person can have, depends on the radii $r_i$. 
By denoting $n_v$ this number, a simple computation shows that 
$$\dsp n_v \leq \dsp \frac{\pi}{
 \arcsin\left( \dsp \frac{r_{\mathrm{min}}}{r_{\mathrm{max}}+r_{\mathrm{min}}}\right)
},$$
where $r_{\mathrm{min}} =\min r_i $ et $r_{\mathrm{max}} =\max r_i$.
With these notations, it can be proved that in the polydisperse case (see~\cite{thesejv}) $$  \dsp \forall \bflambda \in \Lambda_{q, F}, \ 
\forall i<j, \  \lambda_{ij} \leq |\FF|\ b^N  
\textmd{ with } b = \dsp \dsp \frac{2\sqrt{n_v}}{\min\left(\sin \left( \dsp \frac{\pi}{n_v +1}\right) , \sin\left( \dsp \frac{2\pi}{N} \right)\right)}. $$
\end{rmrk}

\begin{rmrk} The radii can be allowed to depend on time in a Lipschitz way. This new problem fits into the framework of sweeping process considered in this paper provided that 
$$ r:= \inf_{t\in [0,T]} \ \inf_{i\in\{1,...,N\}} r_{i}(t) >0.$$
Indeed Proposition \ref{propinegtrianginv} still holds with $I_\rho(t,\qqq)$, $\rho<r$ by replacing $\gamma$ by another constant.
\end{rmrk}

\bibliographystyle{plain}
\bibliography{biblio}

\end{document}